\documentclass{article}
\usepackage[utf8]{inputenc}
\usepackage{xcolor}
\usepackage{amsmath, amsthm, amssymb, amsfonts, mathtools}
\usepackage{graphicx}
\usepackage{caption}
\usepackage{subcaption}
\usepackage{cleveref}
\usepackage{authblk}
\usepackage[a4paper]{geometry}
\usepackage{setspace}
\usepackage[style=numeric-comp,sorting=none,sortcites=true,maxbibnames=15,autopunct=true,autolang=hyphen,hyperref=false,abbreviate=true,backref=false,backend=bibtex]{biblatex}
\usepackage{comment}
\usepackage{tikz}
\usepackage[none]{hyphenat}
\theoremstyle{plain}
\newtheorem{theorem}{Theorem}[section]
\newtheorem{proposition}{Proposition}[section]
\newtheorem{corollary}{Corollary}[section]

\theoremstyle{remark}
\newtheorem{remark}{Remark}
\theoremstyle{definition}
\newtheorem{definition}{Definition}[section]

\title{Stable Chimera States: A Geometric Singular Perturbation Approach.}
\author[1]{Luis Guillermo Venegas-Pineda\thanks{l.g.venegas.pineda@rug.nl}}
\author[1]{Hildeberto Jardón-Kojakhmetov}
\author[2]{Ming Cao}
\affil[1]{\textit{Bernoulli Institute for Mathematics, Computer Science and Artificial Intelligence, University of Groningen, Nijenborgh 9, 9700 AK, Groningen, The Netherlands.}}
\affil[2]{\textit{Engineering and Technology Institute Groningen, University of Groningen, Nijenborgh 9, 9700 AE, Groningen, The Netherlands.}}
\date{}
\begin{document}
\maketitle
\linespread{1.5}
\doublespacing
\vspace{-2.0cm}
\section*{Abstract}
Over the past decades chimera states have attracted considerable attention given their unexpected symmetry-breaking spatio-temporal nature, simultaneously exhibiting synchronous and incoherent behaviours under specific conditions. Despite relevant precursory results of such unforeseen states for diverse physical and topological configurations, there remain structures and mechanisms yet to be unveiled. In this work, using mean-field techniques, we analyze a multilayer network composed by two populations of heterogeneous Kuramoto phase oscillators with coevolutive coupling strengths. Moreover, we employ Geometric Singular Perturbation Theory (GSPT) with the inclusion of a time-scale separation between the dynamics of the network elements and the adaptive coupling strength connecting them, gaining a better insight into the behaviour of the system from a fast-slow dynamics perspective. Consequently, we derive the necessary and sufficient condition to produce stable chimera states  when considering a co-evolutionary intercoupling strength. Additionally, under the aforementioned constraint and with a suitable adaptive law election, it is possible to generate intriguing patterns, such as persistent breathing chimera states. Thereafter, we analyze the geometric properties of the mean-field system with a co-evolutionary intracoupling strength and demonstrate the production of stable chimera states which depend on the associated network. Finally, relaxation oscillations and canard cycles, also related to breathing chimeras, are numerically produced under identified conditions due to the geometry of our system.

\section{Introduction}
\label{Sec:Introduction}
Synchronization is a crucial element in a large variety of natural and practical processes occurring in daily life \cite{Pikovsky2001, Boccaletti2002}. Particularly, synchronization arises in complex systems as neural networks \cite{Uhlhaas2006, Shafiei2020}, circadian rhythms \cite{Yamaguchi2003, Husse2015}, power transmission lines \cite{Rohden2012, Sajadi2022}, flashing patterns in fireflies \cite{Buck1968, Moiseff2010}, superconducting Josepshon junctions \cite{Wiesenfield1998}, chemical oscillations \cite{Taylor2009, Calugaru2020}, neutrino oscillations \cite{Pantaleone1998}, and extensive biological processes \cite{Bick2020, Hauser2022}. A remarkable contribution to this topic was introduced by the Japanese physicist Yoshiki Kuramoto who, inspired by the pioneering work of Arthur Winfree \cite{Winfree1967}, presented an archetypal model for the study of spontaneous synchronization which consists of a complete network of coupled oscillators with randomly distributed natural frequencies and a coupling kernel dependent on the sine of their phase difference \cite{Kuramoto1984}. 
\paragraph*{}Although identical oscillators were expected to only display complete synchronous or incoherent dynamics, Kuramoto and Battogtokh found that, under specific initial conditions, a ring of identical coupled oscillators can separate in two distinguishable spatial regions, one of almost completely synchronized elements whilst the other composed of partially incoherent oscillators \cite{Kuramoto2002, Omelchenko2018}. Later, Abrams and Strogatz named this new unexpected pattern as \textit{chimera state} due to their resemblance to the mythological beast composed of dissonant elements \cite{Abrams2004}. Since then, chimera states have been reported in a vast diversity of contexts, gathered in extensive review papers \cite{Panaggio2015, Scholl2016, Bera2017}, and various physical phenomena have been related to such behavior. For instance, in nature many animal species engage in unihemispheric slow-wave sleep in which a highly active region of the brain coexist with the other hemispheric performing in a more erratic manner \cite{Rattenborg2000}. Similarly, the simultaneous presence of highly synchronized and completely incoherent patterns have been reported in ventricular fibrillation \cite{Davidenko1992, Panfilov1998}, social and cultural trends \cite{Gonzalez2014}, and neurology models such as non-locally coupled Hodgkin-Huxley oscillators \cite{Sakaguchi2006}, leaky integrate-and-fire neurons \cite{Olmi2011}, as well as in coupled FitzHugh-Nagumo oscillators \cite{Omelchenko2013}, among several other neural network models \cite{Panaggio2015}.
\paragraph*{}In this paper, we present a novel mechanism for the generation of chimera states in a complete network composed by different populations of heterogeneous Kuramoto phase oscillators in which the strength of the coupling connecting the nodes adaptively evolves depending on the macroscopic state of the system. Such coevolutionary networks have been observed in several systems like vascular networks \cite{Martens2017}, opinion dynamics \cite{Zino2020, Zino2020-2}, power grids \cite{Berner2021}, and neural processes involved both in learning \cite{Gernster1996, Gernster2002} and the progression of specific neuro-degenerative diseases \cite{Goriely2020}. 
\paragraph*{}Particularly, enabled by our model, we employ the Ott-Antonsen ansatz to obtain a reduced order mean-field representation of the networks' dynamics. Moreover, by introducing a slow adaptive law we induce a fast-slow structure, allowing for the first time the identification of stable chimera states with the stable equilibria of a fast-slow system. \textcolor{black}{Additionally, by analyzing the geometric properties of the mean-field, characteristics leading to different synchronization patterns are determined in the reduced manifold, which then, by considering the normally hyperbolic case, are preserved in the network when assuming sufficiently heterogeneous populations in our ensemble with a slowly coevolutive coupling strength, given a suitable large time-scale separation between the dynamics off and on the network. Hence, we provide a novel mechanism that allows to design the desired synchronization pattern in the mean-field, which is then observed in the network under the adequate conditions.} In particular, we prove that stable chimera states and persistent breathing chimeras are achieved under a sufficiently slow coevolving law with specific parameter arrangements. Finally, other parameter regimes related to the loss of normal hyperbolicity conditions and leading to complex behavior on the reduced system are identified and compared numerically at the network level, generating exciting ideas of future work.
\paragraph*{}The rest of the paper is structured as follows. In section \ref{Sec:Model} we present our research model with all the necessary assumptions, from which we derive an equivalent mean field based on the Ott-Antonsen technique in order to analyze the geometric properties of the system from a reduced order representation. Later, in section \ref{Sec:Analysis}, we present the main results of our study for two distinct problems, the coevolutionary inter and intracoupling, respectively. For each scenario, we determine the geometric properties of the related mean field and give conditions for the effective generation of chimera states when considering adaptive regimes only dependent on macroscopic quantities to preserve the Ott-Antonsen method. Then, utilizing the attractiveness of the Ott-Antonsen manifold, we provide with arguments for the chimera state to be reflected in the finite-size network. Thereafter, in section \ref{Sec:NumericResults} we present numerical results as evidence of our research by comparing the dynamics of the mean field and the network for different macroscopic adaptive laws, \textcolor{black}{showing numerically the successful generation of the desired synchronization patterns by the incorporation of a coevolutive dynamics designed in the mean-field for the coupling strengths of the network.} Finally, in section \ref{Sec:Conclusions} we give our conclusions and discuss future research opportunities as the consequence of the outcomes presented. In particular, we emphasize the analysis of the non-hyperbolic case for which relaxation oscillations and canard solutions, related to breathing chimera states due to the patterns numerically observed, have been produced in simulations for both the mean field and the network. \textcolor{black}{Additionally, in appendix \ref{Sec:Preliminaries} we introduce the basic notions supporting our results, including an overview of the Ott-Antonsen ansatz as well as a brief summary on fast-slow dynamics for the interested reader.}

\section{Coevolutionary multilayer Kuramoto network}
\label{Sec:Model}
The system we analyze consists of a complete network composed by two populations of different heterogeneous Kuramoto oscillators intra and interconnected with common coevolutive coupling strengths, as graphically depicted in figure \ref{FullSystemDiagramLGVP} and mathematically expressed in the following set of coupled ordinary differential equations
\begin{eqnarray}
    \dot{\theta}_{i}^{\sigma} &&= \omega_{i}^{\sigma} + \frac{k_{\sigma}}{N_{\sigma}} \sum_{j=1}^{N_{\sigma}} \sin{\left(\theta_{j}^{\sigma} - \theta_{i}^{\sigma} - \beta_{\sigma \sigma}\right)}+ \frac{\mu}{N_{\sigma'}} \sum_{l=1}^{N_{\sigma'}} \sin{\left(\theta_{l}^{\sigma'} - \theta_{i}^{\sigma} - \beta_{\sigma \sigma'}\right)}\label{GeneralModelLGVP},\\
    \dot{\zeta}_{\sigma \sigma'} &&= \varepsilon g(R_{\sigma}, R_{\sigma'}, \zeta_{\sigma \sigma}, \zeta_{\sigma \sigma'}),\label{GenericAdaptiveLaw}
\end{eqnarray}
where $\theta_{i}^{\sigma}$ and $\omega_{i}^{\sigma}$ represent the phase and natural frequency of the $i$-th oscillator, chosen from a unimodal Cauchy-Lorentz distribution as \eqref{CauchyLorentzPDF}, while $N_{\sigma}$ denotes the number of elements in the $\sigma$-layer ($\sigma=1,2$). \textcolor{black}{Moreover, the parameters $\beta_{\sigma \sigma}$ and $\beta_{\sigma \sigma'}$ represent the phase-lag existing through elements between and across populations, respectively. For the purpose of our work, we consider networks without a phase-lag, i.e. $\beta_{\sigma \sigma'} = 0$, for $\sigma, \sigma' = 1,2$, as the necessary level of heterogeneity in the ensemble is already induced by the natural frequencies. Later, in remark \ref{Rem:PhaseLag}, we present a detailed explanation on why the assumption $\beta_{\sigma \sigma'} = 0$ is considered and its effect on our results. Nevertheless, we display some examples for which $\beta_{\sigma \sigma'}$ is greater than zero but sufficiently small, as it can be treated as a perturbation parameter in our analysis. Finally, in this work we concentrate only on one coevolutive coupling strength at the time, thus we denote an arbitraty coupling strength with the variable $\zeta_{\sigma \sigma'}$, such that the intracoupling strengths are expressed as $\zeta_{\sigma \sigma} = \zeta_{11} = k_{1}$ or $\zeta_{\sigma \sigma} = \zeta_{22} = k_{2}$, while the intercoupling strength is given by $\zeta_{\sigma \sigma'} = \zeta_{12} = \zeta_{21} = \mu$ in \eqref{GenericAdaptiveLaw}. Additionally, the coevolutive coupling strength $\zeta_{\sigma \sigma'}$ is a slowly adaptive variable which only depend on macroscopic quantities of \eqref{GeneralModelLGVP}}, simplifying the reduction method by considering the couplings as constant parameters in the fast time-scale and avoiding modifications on the mean-field technique employed \cite{Ciszak2020}. 
\begin{figure}[ht]
    \centering
    \includegraphics[width=0.6\textwidth]{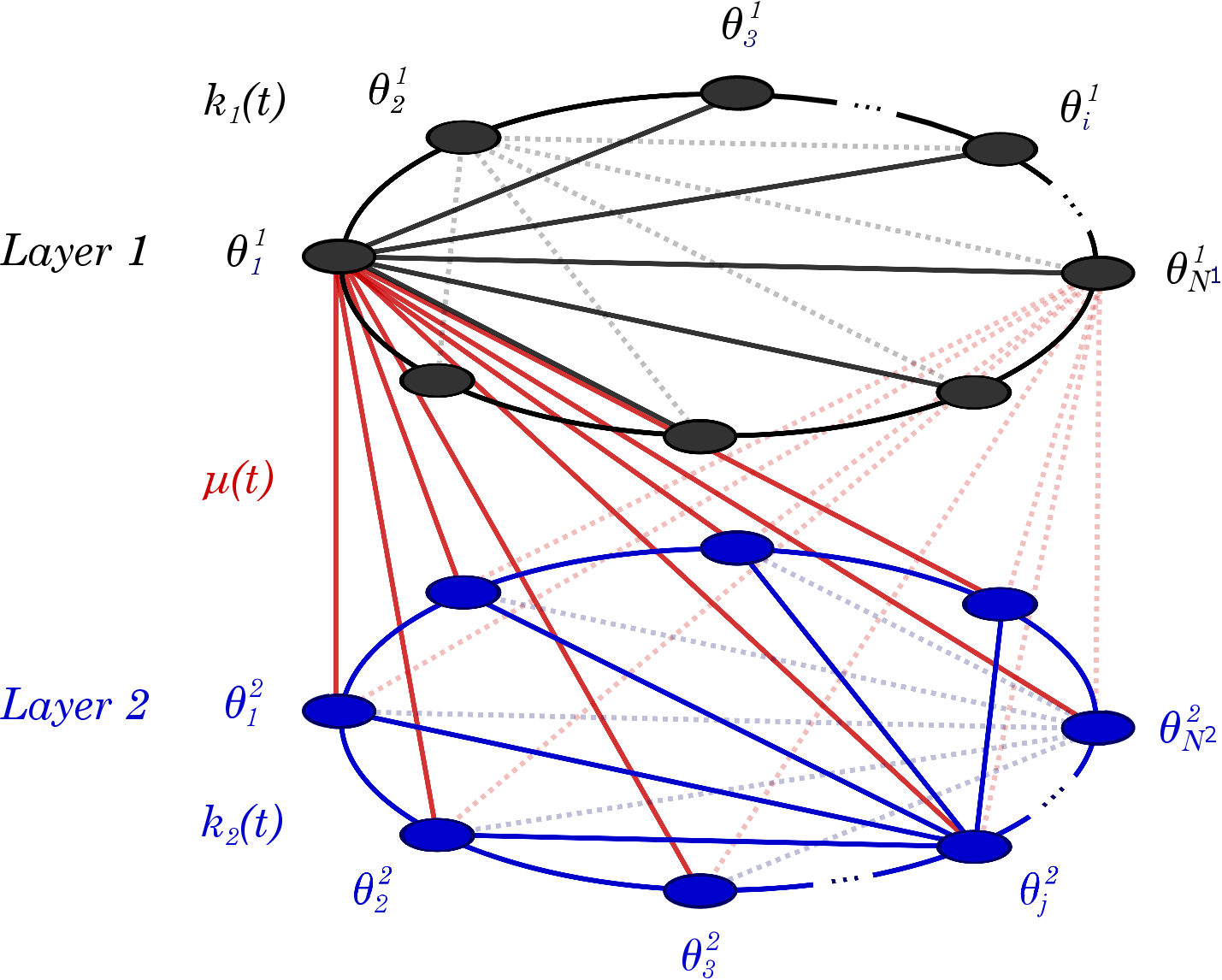}
    \caption{Graphic representation of the multilayer coevolutive network \eqref{GeneralModelLGVP}. The black and blue components represent elements corresponding to the first and second populations, respectively. Moreover, $\theta_{i}^{\sigma}$ denotes the phase of each element in the network, with $i=1,\dots, N_{\sigma}$ and $\sigma=1,2$. The adaptive intercoupling strength (red) is expressed as $\mu(t)$ and the intracoupling variable for each population is captured by $k_{\sigma}$. Although \eqref{GeneralModelLGVP} is a complete graph, we present only certain nodes' connections for simplicity.}
    \label{FullSystemDiagramLGVP}
\end{figure}
Thus, in accordance to \eqref{OttAntonsenRealMeanField}, \textcolor{black}{when $\beta_{\sigma \sigma'} = 0$, with  $\sigma, \sigma' = 1,2$}, a mean-field representation of \eqref{GeneralModelLGVP} is given by
\begin{align}
\begin{split}
    \dot{\rho}_{\sigma} &= -\Delta_{\sigma}\rho_{\sigma} + \frac{1}{2}\left(1-\rho_{\sigma}^{2}\right)\left(k_{\sigma}\rho_{\sigma} + \mu \rho_{\sigma'}\cos{(\phi_{\sigma'}-\phi_{\sigma})}\right),\\
    \dot{\phi}_{\sigma} &= -\omega_{\sigma} + \frac{1}{2}\mu\frac{\rho_{\sigma}^{2}+1}{\rho_{\sigma}}\rho_{\sigma'}\sin{(\phi_{\sigma'}-\phi_{\sigma})},\\
    \textcolor{black}{\dot{\zeta}_{\sigma \sigma'}} &= \varepsilon g(\rho_{\sigma}, \rho_{\sigma'}, \zeta_{\sigma \sigma'}).
    \label{OAReducedLGVP}
\end{split}
\end{align}
Furthermore, considering our two-population problem and by introducing the additional variables $\psi = \phi_{2}-\phi_{1}$, and $\Omega=\omega_{2}-\omega_{1}$, the mean-field \eqref{OAReducedLGVP} is expressed as
\begin{align}
\begin{split}
    \dot{\rho}_{1} &= -\Delta_{1}\rho_{1} + \frac{1}{2}\left(1-\rho_{1}^{2}\right)\left(k_{1}\rho_{1} + \mu \rho_{2}\cos{\psi}\right),\\
    \dot{\rho}_{2} &= -\Delta_{2}\rho_{2} + \frac{1}{2}\left(1-\rho_{2}^{2}\right)\left(k_{2}\rho_{2}+\mu \rho_{1}\cos{\psi}\right),\\
    \dot{\psi} &= -\Omega-\frac{1}{2}\mu\left(\frac{\rho_{1}^{2}+\rho_{2}^{2}+2\rho_{1}^{2}\rho_{2}^{2}}{\rho_{1}\rho_{2}}\right)\sin{\psi},\\
    \textcolor{black}{\dot{\zeta}_{\sigma \sigma'}} &=\varepsilon g\left( \rho_{\sigma}, \rho_{\sigma'}, \zeta_{\sigma \sigma}, \zeta_{\sigma \sigma'} \right).
    \label{FullReducedSystemLGVP}
\end{split}
\end{align}

\begin{remark}
    Observe that, since the selected probability distribution functions are of unimodal Cauchy-Lorentz type, the real part of the local Kuramoto order parameter $R_{\sigma}$ of \eqref{GeneralModelLGVP} directly corresponds to the variables $\rho_{\sigma}$ in \eqref{OAReducedLGVP}. 
\end{remark}
In the following section we analyze the behaviour of \eqref{FullReducedSystemLGVP} under the presence of different types of adaptive coupling strengths, only dependent on macroscopic quantites of \eqref{GeneralModelLGVP}. In particular, we benefit from the introduction of a time-scale separation between the dynamics on the nodes and that related to the coupling strength, allowing us the use of fast-slow theory without alterations on the reduction method employed.

\section{Analysis}
\label{Sec:Analysis}
To begin with and in order to facilitate computations, we reduce the dimension of the fast-time mean field \eqref{FullReducedSystemLGVP} by demonstrating that a near synchronization state is attracting when certain parameter conditions are fulfilled. The aforementioned idea is synthesized in the next proposition.
\begin{proposition}[Invariant and attracting synchronization set]
    \label{LemmaRho2}
        The set $\left\{\left( \rho_{1}, \rho_{2}, \psi, \mu \right)\in\mathbb{K} : \rho_{2}=1\right\}$ in \eqref{FullReducedSystemLGVP},  $\left(\mathbb{K}=[0,1]^{2} \times [0,2\pi) \times \mathbb{R}\right)$, is invariant and attracting if and only if $\Delta_{2} = 0$, for $k_{2}$ large enough. Moreover, there exists an invariant and attracting set $\left\{\left( \rho_{1}, \rho_{2}, \psi, \mu \right)\in\mathbb{K} : \rho_{2}(\Delta_{2})=1-a\Delta_{2}+O\left(\Delta_{2}^{2}\right)\right\}$ with $a>0$, and $\Delta_{2}$ positive and sufficiently small when $k_{2}>0$ is adequately large.
    \end{proposition}
    \begin{proof}
       Consider the function
       \begin{equation}
           h_{2}(\rho_{2},\Delta_{2})\coloneqq-\Delta_{2}\rho_{2}+\frac{1}{2}\left( 1-\rho_{2}^{2} \right)\left( k_{2}\rho_{2} + \mu\rho_{1}\cos{\psi} \right),
           \label{FlowRho2}
       \end{equation}
       associated with the flow equation of $\rho_{2}$ in \eqref{FullReducedSystemLGVP}. Thus, the set $\left\{\left( \rho_{1}, \rho_{2}, \psi, \mu \right)\in\mathbb{K} : \rho_{2}=1\right\}$ is invariant if and only if $\Delta_{2}=0$. Moreover, by taking the derivative of \eqref{FlowRho2} with respect to $\rho_{2}$ and evaluating at $(\rho_{2}, \Delta_{2}) = (1,0)$ yields
       \begin{equation}\label{AttractingSetP2}
           \frac{\partial h_{2}}{\partial \rho_{2}}\bigg\rvert_{(1,0)} = -\left( k_{2} + \mu\rho_{1}\cos{\psi} \right),
       \end{equation} 
       which is strictly negative in $\mathbb{K}$ for $k_{2}$ adequately large. Therefore, from the implicit function theorem there exists a set $\{ \left(\rho_{1}, \rho_{2}, \psi, \mu \right)\in\mathbb{K}:\rho_{2} = q(\Delta_{2}) \}$  such that $h_{2}\left(q(\Delta_{2}), 
       \Delta_{2}\right)=0$, with $q(0)=1$, $q\in C^{1}$. In particular, we obtain an expansion of $\rho_{2}$ as a regular perturbation on $\Delta_{2}$ in the form $\rho_{2}(\Delta_{2})=1-a\Delta_{2}+O\left(\Delta_{2}^{2}\right)$ with the invariant flow of \eqref{FlowRho2} given now as
       \begin{equation}
           0 = \left( -1 + a k_{2} + a\mu\rho_{1}\cos{\psi} \right)\Delta_{2} + O\left(\Delta_{2}^{2}\right),
           \label{ExpansionRho2}
       \end{equation}
       which is satisfied for $a=\left( k_{2} + \mu\rho_{1}\cos{\psi} \right)^{-1}$, positive whenever $|k_{2}|\geq|\mu|$ along the interval $\rho_{1}\in(0,1)$. Moreover, the derivative of \eqref{FlowRho2} evaluated at $\rho_{2}\left( \Delta_{2} \right) = 1-a\Delta_{2}+ O\left( \Delta_{2}^{2} \right)$, given by
       \begin{eqnarray}
           \frac{\partial h_{2}}{\partial \rho_{2}}\bigg\rvert_{\rho_{2}=1-a\Delta_{2}+O\left(\Delta_{2}^{2}\right)} = -\Delta_{2} +\frac{1}{2}k_{2}\left( a\Delta_{2}(2-3a\Delta_{2})-2 \right)- \mu\rho_{1}(1-a\Delta_{2})\cos{\psi},
           \label{DerivativeExpansion}
       \end{eqnarray}
      approaches \eqref{AttractingSetP2} in the limit when $\Delta_{2}$ tends to zero, rendering the set $\{ \left(\rho_{1}, \rho_{2}, \psi, \mu \right)\in\mathbb{K}:\rho_{2}(\Delta_{2}) = 1-a\Delta_{2}+O(\Delta_{2}^{2})\}$ attracting with $\Delta_{2}$ sufficiently small for $k_{2}$ sufficiently large.
    \end{proof}
    Therefore, from proposition \ref{LemmaRho2} and by choosing $k_{2}$ sufficiently large and $\Delta_{2}$ sufficiently small, in what follows we simply consider $\rho_{2}=1$. Moreover, we analyze two different problems, namely, the co-evolutionary intercoupling, i.e., $\mu$ adaptive, and the coevolutive intracoupling, i.e., $k_{1}$ adaptive. This distinction is made in order to compare the geometries of both adaptation schemes and their isolated effects in the network. Notice that stable chimera states can be generated in both coevolving scenarios with determined geometric differences, as we subsequently demonstrate. To begin with, in the next subsection we study the adaptive intercoupling problem in \eqref{FullReducedSystemLGVP} while considering the intracouplings as constant parameters.

\subsection{Co-evolutionary Intercoupling}
    As a first approach, the coevolving intercoupling strength $\mu$ effects are analyzed. Therefore, by fixing $\rho_{2}=1$ and $k_1, k_2\in\mathbb R$, with $k_{2}$ sufficiently large, as parameters in both populations of \eqref{GeneralModelLGVP}, the mean-field \eqref{FullReducedSystemLGVP} is
     \begin{align}
        \begin{split}
            \dot{\rho}_{1} &= -\Delta_{1}\rho_{1} + \frac{1}{2}\left(1-\rho_{1}^{2}\right)\left(k_{1}\rho_{1} + \mu \cos{\psi}\right),\\
            \dot{\psi} &= -\Omega - \frac{1}{2}\mu\left(\frac{3\rho_{1}^{2}+1}{\rho_{1}}\right)\sin{\psi},\\
            \dot{\mu} &= \varepsilon_{\mu}f\left( \rho_{1}, \psi, \mu, t \right),      \label{CoevolutiveIntercouplingReducedLGVP}
        \end{split}
    \end{align}
    where $f(\rho_{1}, \psi, \mu, t)$ is some adaptive law which only depends on macroscopic quantities of \eqref{GeneralModelLGVP}, with $\varepsilon_{\mu}$ as the time-scale separation parameter. Consequently, for the critical manifold of \eqref{CoevolutiveIntercouplingReducedLGVP} we obtain
    \begin{align*}
    \begin{split}
        \mu \cos{\psi} &= \rho_{1}\left( \dfrac{2\Delta_{1}}{1-\rho_{1}^{2}} - k_{1} \right),\\
        \mu \sin{\psi} &= -2\Omega\left( \frac{\rho_{1}}{3\rho_{1}^{2}+1} \right).
    \end{split}
    \end{align*}
Hence, the critical manifold consists of two branches defined by
\begin{eqnarray} \mathcal{C}_{0}^{\pm} = \Bigg\{ \left( \rho_{1}, \psi, \mu \right) \in \mathbb{L} :
    \mu = \pm \rho_{1}\sqrt{\bigg( \frac{2\Delta_{1}}{1-\rho_{1}^{2}} - k_{1} \bigg)^{2} + \bigg( \frac{2\Omega}{3\rho_{1}^{2}+1} \bigg)^{2}} \Bigg\},
    \label{CoevolutiveInterCriticalManifold}
\end{eqnarray}
with $\mathbb{L}=  \left[0,1\right)\times\left[0,2\pi\right)\times \mathbb{R}$. Moreover, the Jacobian of \eqref{CoevolutiveIntercouplingReducedLGVP} evaluated at \eqref{CoevolutiveInterCriticalManifold} is

\begin{equation}
    J_{f}\Big\rvert_{\mathcal{C}_{0}^{\pm}} = \begin{bmatrix}
    -\Delta_{1}\left( \frac{1+\rho_{1}^{2}}{1-\rho_{1}^{2}} \right)+\frac{1}{2}k_{1}\left(1-\rho_{1}^{2}\right) & \Omega\rho_{1}\left(\frac{1-\rho_{1}^{2}}{1+3\rho_{1}^{2}} \right)\\
    \frac{\Omega}{\rho_{1}}\left( \frac{3\rho_{1}^{2}-1}{3\rho_{1}^{2}+1} \right) & -\frac{1}{2}\left(\frac{1+3\rho_{1}^{2}}{1-\rho_{1}^{2}}\right)\left( 2\Delta_{1}-k_{1}(1-\rho_{1}^{2}) \right)
    \end{bmatrix}.
    \label{JacobianMatrixCriticalManifold}
\end{equation}

Different results are obtained when considering if the distributions from which the natural frequencies are drawn from are centered at the same value or not, i.e., $\Omega=0$ or $\Omega\neq0$, respectively. We start by analyzing the former case in the following proposition.
\begin{proposition}[Intercoupling normal hyperbolicity and stability ($|\Omega|\ll1$)]
    \label{PropositionInterOmega=0}
    The critical manifold \eqref{CoevolutiveInterCriticalManifold} of \eqref{CoevolutiveIntercouplingReducedLGVP}, with $\Omega=0$, is normally hyperbolic in the interval $\rho_{1}\in[0,1)$ if and only if $k_{1}<2\Delta_{1}$. Moreover, if the critical manifold is normally hyperbolic then it is attracting for every $\rho_{1}\in[0,1)$. Finally, for $|\Omega|\ll1$, the normal hyperbolicity and stability of the $\Omega=0$ case are conserved. 
\end{proposition}
\begin{proof}
    Notice that, when $\Omega=0$, the Jacobian \eqref{JacobianMatrixCriticalManifold} reduces to 

\begin{equation}
    J_{f}\Big\rvert_{\mathcal{C}_{0}^{\pm}} = \begin{bmatrix}
    -\Delta_{1}\left( \frac{1+\rho_{1}^{2}}{1-\rho_{1}^{2}} \right)+\frac{1}{2}k_{1}\left(1-\rho_{1}^{2}\right) & 0\\
    0 & -\frac{1}{2}\left(\frac{1+3\rho_{1}^{2}}{1-\rho_{1}^{2}}\right)\left( 2\Delta_{1}-k_{1}(1-\rho_{1}^{2}) \right)
    \end{bmatrix}.
    \label{JacobianMatrixCriticalManifoldOmega=0}
\end{equation}

    Thus, the eigenvalues of \eqref{JacobianMatrixCriticalManifoldOmega=0} are given as
    \begin{align*}
        \lambda_{1} &= -\Delta_{1}\left( \frac{1+\rho_{1}^{2}}{1-\rho_{1}^{2}} \right) + \frac{1}{2}k_{1}\left(1-\rho_{1}^{2}\right),\\
        \lambda_{2} &= -\frac{1}{2}\left( \frac{1+3\rho_{1}^{2}}{1-\rho_{1}^{2}} \right)\left(2\Delta_{1}-k_{1}(1-\rho_{1}^{2})\right),
    \end{align*}
    which are purely real for all $\rho_{1}\in[0,1)$. Consequently, for normal hyperbolicity we require that
    \begin{align*}
        \lambda_{1} &\neq0 \Rightarrow 1>\frac{\left( 1-\rho_{1}^{2} \right)^{2}}{1+\rho_{1}^{2}}\neq\frac{2\Delta_{1}}{k_{1}},\\
        \lambda_{2} &\neq0 \Rightarrow 1>\left(1-\rho_{1}^{2}\right) \neq \frac{2\Delta_{1}}{k_{1}}.
    \end{align*}
    Therefore, the critical manifold \eqref{CoevolutiveInterCriticalManifold} is normally hyperbolic in the interval $\rho_{1}\in[0,1)$ if and only if $k_{1}<2\Delta_{1}$. Moreover, for stability notice that if $k_{1}<2\Delta_{1}$ then $\lambda_{1,2}<0$ for every $\rho_{1}\in[0,1)$. Finally, if now we consider $|\Omega|$ positive and suffiently small the normal hyperbolicity and attractiveness of the $\Omega=0$ case are persistent as the eigenvalues $\lambda_{1,2}$ of \eqref{JacobianMatrixCriticalManifold} depend differentiably on $\Omega$ \cite{Lax2007}.
\end{proof}
Hence, proposition \ref{PropositionInterOmega=0} states that as long as the natural frequencies of the first population are sufficiently sparse, every point in the critical manifold \eqref{CoevolutiveInterCriticalManifold} is normally hyperbolic and attracting in the interval $\rho_{1}\in(0,1)$, when the centers of the natural frequency distributions are sufficiently close. Observe that this conclusion matches with the global attractiveness condition of the Ott-Antonsen manifold discussed in the introduction of this work.
\paragraph*{}Next, we study the case when only the intracoupling of one population evolves on time while preserving constant both the other intracoupling and the intercoupling strength. For convenience, we analyze system \eqref{FullReducedSystemLGVP}, under proposition \ref{LemmaRho2}, for $k_{1}$ adaptive.

\begin{remark}
    It is important to highlight that the identification of two branches in \eqref{CoevolutiveInterCriticalManifold}, namely $\mathcal{C}_{0}^{+}$ and $\mathcal{C}_{0}^{-}$, is done for convenience of notation. Nevertheless, by definition \ref{CriticalManifoldDefinition}, every point in these branches conforms the critical manifold, regardless its location in a particular segment, as shown in figures \ref{Fig:InterFeedback} and \ref{Fig:IntraFeedback} for the coevolutionary inter and intracoupling, respectively.
\end{remark}

\subsection{Coevolutionary Intracoupling}
As a comparison to the adaptive intercoupling case of the previous section, we now analyze the effect of one coevolutive intracoupling. By letting $k_{1}$ adaptive and with the result in proposition \ref{LemmaRho2}, now \eqref{FullReducedSystemLGVP} reduces to
    \begin{align}
    \begin{split}
        \dot{\rho}_{1} &= -\Delta_{1}\rho_{1} + \frac{1}{2}\left(1-\rho_{1}^{2}\right)\left( k_{1}\rho_{1} + \mu\cos{\psi}\right),\\
        \dot{\psi} &= -\Omega-\frac{1}{2}\mu\left( \frac{3\rho_{1}^{2}+1}{\rho_{1}} \right)\sin{\psi},\\
        \dot{k}_{1} &= \varepsilon_{1}f\left( \rho_{1}, \psi, k_{1}, t \right),
        \label{CoevolutiveIntraReducedLGVP}
    \end{split}
    \end{align}
with the same variable interpretations as before and $k_{1}$ as the coevolutive coupling strength. Thus, the critical manifold of \eqref{CoevolutiveIntraReducedLGVP} is
\begin{eqnarray} \mathcal{C}_{0}^{\pm} = \Bigg\{ \left( \rho_{1}, \psi, k_{1} \right) \in \mathbb{L} :
    k_{1} = \frac{2\Delta_{1}}{1-\rho_{1}^{2}}\pm \sqrt{\bigg(\frac{\mu}{\rho_{1}}\bigg)^{2} - \bigg
    (\frac{2\Omega}{3\rho_{1}^{2}+1}  \bigg)^{2}} \Bigg\}.
    \label{CoevolutiveIntraCriticalManifold}
\end{eqnarray}
\begin{remark}
    Notice that \eqref{CoevolutiveIntraCriticalManifold} can have complex elements, producing disconnected branches of $\mathcal{C}_{0}^{\pm}$. In fact, by definition \ref{CriticalManifoldDefinition} the critical manifold is a geometric object of real components. Therefore, we restrict our attention to the case for which the discriminant in \eqref{CoevolutiveIntraCriticalManifold} is non-negative.
\end{remark}
\begin{proposition}[Connected critical manifold]\label{RealConnectedBranches}
    The critical set \eqref{CoevolutiveIntraCriticalManifold} of \eqref{CoevolutiveIntraReducedLGVP} is completely connected if and only if $|\mu| \geq |\Omega|/\sqrt{3}$.
\end{proposition}    
\begin{proof}
    In order to produce only real components we require that
    \begin{equation*}\label{ProofRealConnectedBranches}
        \left( \frac{\mu}{\rho_{1}} \right)^{2} - \left( \frac{2\Omega}{3\rho_{1}^{2}+1} \right)^{2} \geq 0,
    \end{equation*}
    which can be rewritten as
    \begin{equation}
    \label{ProofConnectedCritManifoldIntra}
        \frac{|\mu|}{2|\Omega|}\geq\frac{\rho_{1}}{3\rho_{1}^{2}+1}.
    \end{equation}
    By calculating the derivative of the right-hand side of  \eqref{ProofConnectedCritManifoldIntra} with respect to $\rho_{1}$, its maximum occurs at $\rho_{1}=1/\sqrt{3}$, for which the function reaches a value of $1/2\sqrt{3}$. Hence, the critical manifold \eqref{CoevolutiveIntraCriticalManifold} is completely connected if and only if $\left|\mu\right| \geq |\Omega|/\sqrt{3}$.
\end{proof}

\begin{corollary}
    When $\Omega=0$, the critical manifold \eqref{CoevolutiveIntraCriticalManifold} is completely connected for every $\mu\in\mathbb{R}$.
\end{corollary}
Next, we obtain the Jacobian of \eqref{CoevolutiveIntraReducedLGVP} with respect to the fast variables, which by evaluation at \eqref{CoevolutiveIntraCriticalManifold} yields

\begin{equation}
    J_{f}\Big\rvert_{\mathcal{C}_{0}^{\pm}} = \begin{bmatrix}
        -2\Delta_{1}\left( \frac{\rho_{1}^{2}}{1-\rho_{1}^{2}}\right) \pm \frac{1}{2}\left(1-\rho_{1}^{2}\right) h\left(\rho_{1}\right) & \Omega\rho_{1}\left( \frac{1-\rho_{1}^{2}}{3\rho_{1}^{2}+1}\right)\\ \frac{\Omega}{\rho_{1}}\left( \frac{3\rho_{1}^{2}-1}{3\rho_{1}^{2}+1} \right) & \pm\frac{1}{2}(3\rho_{1}^{2}+1)h\left(\rho_{1}\right)
    \end{bmatrix},
    \label{JacobianMatrixCriticalManifoldIntra}
\end{equation}

with $h(\rho_{1})\coloneqq \sqrt{\left(\frac{\mu}{\rho_{1}}\right)^{2} - \left(\frac{2\Omega}{3\rho_{1}^{2}+1} \right)^{2}}$, non-negative under proposition \ref{RealConnectedBranches}. Thus, the following result summarize the conditions for the normal hyperbolicity and stability of \eqref{CoevolutiveIntraCriticalManifold}.
\begin{proposition}[Intracoupling normal hyperbolicity and stability $(|\Omega|\ll1)$]\label{PropIntraOmega0}
    The critical manifold \eqref{CoevolutiveIntraCriticalManifold} of \eqref{CoevolutiveIntraReducedLGVP} has one normally hyperbolic and attracting branch  $\left(\mathcal{C}_{0}^{-}\right)$, and one unstable branch $\left(\mathcal{C}_{0}^{+}\right)$ along the interval $\rho_{1}\in(0,1)$ for every $\mu\neq0$, with $\Omega=0$. Moreover, for $|\Omega|\ll1$, the normal hyperbolicity and stability conditions of the $\Omega=0$ case are preserved.
\end{proposition}
\begin{proof}
    Observe that for $\Omega=0$ the Jacobian \eqref{JacobianMatrixCriticalManifoldIntra} reduces to
\begin{equation*}
    J_{f}\Big\rvert_{\mathcal{C}_{0}^{\pm}} = \begin{bmatrix}
        -2\Delta_{1}\left( \frac{\rho_{1}^{2}}{1-\rho_{1}^{2}}\right) \pm \frac{1}{2}\left(\frac{1-\rho_{1}^{2}}{\rho_{1}} \right) |\mu| & 0 \\ 0 & \pm \frac{1}{2}\left(\frac{3\rho_{1}^{2}+1}{\rho_{1}}\right)|\mu|
    \end{bmatrix},
    \label{JacobianMatrixCriticalManifoldIntraOmegaEqual}
\end{equation*}
with eigenvalues in the form
\begin{subequations}
    \begin{align*}
        \lambda_{1}^{\pm} &= -2\Delta_{1}\left( \frac{\rho_{1}^{2}}{1-\rho_{1}^{2}} \right) \pm \frac{1}{2}\left( \frac{1-\rho_{1}^{2}}{\rho_{1}} \right)\left| \mu \right|,\\
        \lambda_{2}^{\pm} &= \pm\frac{1}{2}\left( \frac{3\rho_{1}^{2}+1}{\rho_{1}} \right)\left| \mu \right|.
    \end{align*}
\end{subequations}
Hence, the branch of \eqref{CoevolutiveIntraCriticalManifold} corresponding to $\lambda_{1,2}^{-}$, namely $\mathcal{C}_{0}^{-}$, is normally hyperbolic and attracting along the interval $\rho_{1}\in(0,1)$ for every $\mu\neq0$. On the other hand, the branch $\mathcal{C}_{0}^{+}$ is either of saddle-type or repelling. Finally, for $|\Omega|\ll1$ the normal hyperbolicity and stability conditions of the $\Omega=0$ case persist since the eigenvalues $\lambda_{1,2}^{\pm}$ of \eqref{JacobianMatrixCriticalManifoldIntra} depend differentiably on $\Omega$ \cite{Lax2007}.
\end{proof}

    \textcolor{black}{
    \begin{remark}
        The implementation of proposition \ref{LemmaRho2} in system \eqref{FullReducedSystemLGVP} sets the problem into a convenient scenario for chimera-like behaviors, in which one population is synchronized whilst the other can produce different patterns, as similar procedures have been applied in the non-adaptive case \cite{Abrams2008}. Nevertheless, if such a result is neglected, the normal hyperbolicity and stability properties of the critical manifolds \eqref{CoevolutiveInterCriticalManifold} and \eqref{CoevolutiveIntraCriticalManifold}, for both the coevolutive inter and intracoupling cases, \eqref{CoevolutiveIntercouplingReducedLGVP} and \eqref{CoevolutiveIntraReducedLGVP} respectively,  will be different, although the dimension of such manifolds is preserved as long as only one slowly adapting component is considered given the fast-slow structure of the problem.
    \end{remark}}

Consequently, since we have already provided with the conditions for normal hyperbolicity and attractiveness of the critical manifolds \eqref{CoevolutiveInterCriticalManifold} and \eqref{CoevolutiveIntraCriticalManifold} of the adaptive inter and intracoupling systems, now we present the result which enables us to produce different synchronization patterns, including stable chimera states, by considering a slow macroscopic coevolution in \eqref{CoevolutiveIntercouplingReducedLGVP} and \eqref{CoevolutiveIntraReducedLGVP}.
\begin{theorem}[Stable chimera states on the mean field]
    Given an adequate macroscopic adaptive law in the coevolving intercoupling \eqref{CoevolutiveIntercouplingReducedLGVP} or intracoupling \eqref{CoevolutiveIntraReducedLGVP} case, any long-time synchronization pattern, including stable chimera states, is produced in the perturbed mean field when the normal hyperbolicity and attracting conditions for the critical manifold $\mathcal{C}_{0}^{\pm}$ are satisfied, for $\varepsilon$ positive and sufficiently small.
    \label{TheoremNetworkToMeanField}
\end{theorem}
\begin{proof}
    Once we have determined the normally hyperbolic and attracting conditions of the critical manifold $\mathcal{C}_{0}^{\pm}$, and recalling that this geometric object is composed by exponentially stable equilibria of the fast problem, we can construct a suitable slow macroscopic adaptive law with an exponentially stable equilibrium point at $\left( \rho_{1}, \psi, \textcolor{black}{\zeta_{\sigma \sigma'}} \right) = \left( \rho_{1}^{*}, \psi^{*}, \textcolor{black}{\zeta_{\sigma \sigma'}}^{*} \right)$, such that if this point is part of $\mathcal{C}_{0}^{\pm}$, then it is exponentially stable in the perturbed system \cite{Kokotovic1999}, for \textcolor{black}{$\zeta_{\sigma \sigma'}$, with $\sigma, \sigma' = 1,2$}, denoting the type of coevolutive coupling considered, either inter \textcolor{black}{$\left( \zeta_{\sigma \sigma'} = \zeta_{12} = \zeta_{21} = \mu \right)$, or intracoupling $\left( \zeta_{11} = k_{1}, \zeta_{22} = k_{2} \right)$}. Moreover, by Fenichel's theory there exists a slow manifold for which the dynamics converge to the slow flow as $\varepsilon\rightarrow0$, and by regular perturbation results the dynamics of the unperturbed system persist under weak perturbations, so there exists a point in a neighborhood of $\left( \rho_{1}^{*}, \psi^{*}, \textcolor{black}{\zeta_{\sigma \sigma'}}^{*} \right)$ which is an exponentially stable equilibrium of the perturbed system \eqref{FullReducedSystemLGVP} for sufficiently small $\varepsilon$. Finally, recalling that we consider a near synchronization second population, i.e., $\rho_{2}\rightarrow1$, it is possible to set the equilibrium point at any position along the critical manifold with an adequate adaptive law. Hence, by persistence any synchronization behavior, including stable chimera states, can be produced in the mean-field of the weakly perturbed system.
\end{proof}
Finally, the next result provides the necessary condition which allows us to identify the dynamics occurring in the mean-field system with long time behaviours at the network level.
\begin{theorem}[Stable chimera states on a network]    
    \label{TheoremNetworkMeanField}
    Under normally hyperbolic conditions of the critical manifold $\mathcal{C}_{0}^{\pm}$, both for the adaptive intercoupling \eqref{CoevolutiveIntercouplingReducedLGVP} or intracoupling \eqref{CoevolutiveIntraReducedLGVP} systems, dynamical behaviours observed in the mean-field reduction, such as stable chimera states and sustained breathing chimera states, are preserved in the network when considering heterogeneous populations with a slowly coevolving and macroscopic coupling strength, for $\varepsilon$ sufficiently small and for $N$ large enough.
\end{theorem}
\begin{proof}
        It is known that, as long as the probability distributions for the natural frequencies of the oscillators have a finite width, i.e. $\Delta_{\sigma}>0$ $(\sigma = 1,2,\dots, M)$, the attractors for the order parameter dynamics obtained in the reduced Ott-Antonsen manifold \eqref{OAReducedLGVP} are the only attractors of the network \eqref{GeneralModelLGVP} \cite{OttAntonsen2009, OttAntonsen2011}. Moreover, theorem \ref{TheoremNetworkToMeanField} demonstrates the preservation of equilibrium points on the perturbed system for $\varepsilon$ small enough. Therefore, dynamical behaviours observed in the mean-field reduction, including stable chimera states and breathing chimera states, are present in the network when considering non-homogeneous populations.
\end{proof}
In the following section, we show the successful generation of stable chimera states by following our results, with numerical simulations for both the network \eqref{GeneralModelLGVP} and the mean-field \eqref{FullReducedSystemLGVP} when considering the inter and intracoupling adaptive scenarios discussed in this section.

\section{Numerical results}
\label{Sec:NumericResults}
In this section we present a series of simulations validating the effectiveness of our results as well as some intriguing patterns observed under certain non-hyperbolic conditions. It is important to emphasize though that the Ott-Antonsen ansatz consider the thermodynamic limit, i.e. $N_{1,2}\rightarrow\infty$, so the network and the mean-field outputs tend to increasingly approach as the number of oscillators is augmented. For this reason, the images corresponding to the network results have two components, the actual output (orange) and a smooth version (red) produced by filtering the original trajectory with a Savitzky-Golay filter \cite{SavitskyGolay1964}, as this least-squares smoothing method reduces noise while preserving the shape and height of waveform peaks \cite{Schafer2011}, producing a closer approximation to the mean-field response without dramatically increasing the number of oscillators in our network. Additionally, for the numeric results presented in this work we adopt the observations made by B{\"o}hle, Kuehn and Thalhammer, which as they claim makes the simulations faster compared to classical approaches \cite{Bohle2022}.
    \subsection{Coevolutionary Intercoupling}
    First, the behavior of both the network \eqref{GeneralModelLGVP} and its mean-field counterpart \eqref{CoevolutiveIntercouplingReducedLGVP} are compared when different intercoupling coevolving rules are employed. In figure \ref{Fig:InterFeedback}, we present projections to the $\left(R_{1},\mu\right)$ and $\left(\rho_{1},\mu\right)$ planes of \eqref{GeneralModelLGVP} and \eqref{CoevolutiveIntercouplingReducedLGVP}, with $\dot{\mu}=\varepsilon_{\mu}(-\mu + \gamma_{\mu} - \eta_{\mu}\rho_{1})$. Notice that $k_{1}<2\Delta_{1}$ and thus the critical manifold (grey) is normally hyperbolic and attracting for every $\rho_{1}\in(0,1)$. Moreover, observe that the behavior of the order parameters $R_{1}$ and $\rho_{1}$ (red) converge to the intersection of the critical manifold and the equilibria set of the slow adaptation (magenta), which corresponds to a highly incoherent region of the first layer and with the assumption that the second population is near synchrony, i.e., $\rho_{2}\approx1$ and $R_{2}\approx1$, we have effectively produced an exponentially stable chimera state, demonstrating the generation of such pattern in the network \eqref{GeneralModelLGVP} by geometric results derived from its mean-field reduction \eqref{CoevolutiveIntercouplingReducedLGVP}, as stated by theorem \ref{TheoremNetworkMeanField}.
    \paragraph*{}Similarly, in figure \ref{Fig:InterBreathing} we present a persistent breathing chimera state produced through the same mechanism as before in both \eqref{GeneralModelLGVP} and \eqref{CoevolutiveIntercouplingReducedLGVP}, with $\dot{\mu} = \varepsilon_{\mu}\cos{(0.02 t)}$. Observe that, since the fast structure of the problem has not been modified, the critical manifold, as well as its normal hyperbolicity and stability conditions, remain the same and notice that the order parameters $R_{1}$ and $\rho_{1}$ oscillate between high and low synchrony regimes, which correspond to the so called breathing chimera state, characterized by sustained periodical variations on the synchronization level \cite{Omelchenko2022}. Finally, notice that the synchronization level of the second population presents small-amplitude oscillations with the same frequency of the first population due to the intracoupling strength $k_{2}$ employed, when this value is increased a stronger interaction in the second population is achieved and the oscillatory pattern is only appreciated in the first population. Therefore, the ratio between both intracoupling strengths $k_{1,2}$ is selected in such a way that the oscillations in the second populations' level of synchrony exhibits only small amplitude oscillations.

    \begin{figure*}[htbp]
        \begin{subfigure}[htbp]{0.485\textwidth}
            \includegraphics[width=1.0\linewidth]{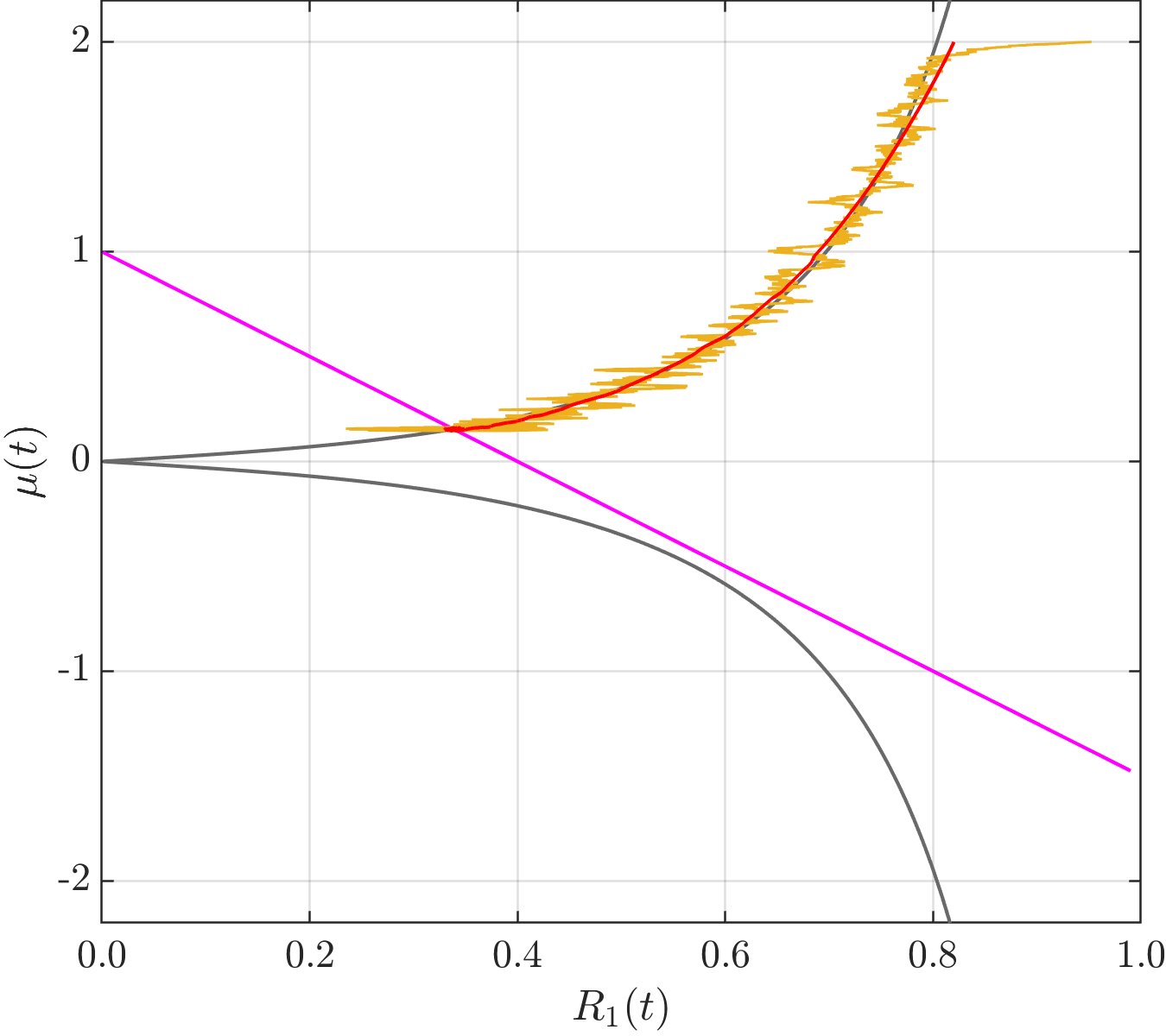}
        \end{subfigure}\hfill
        \begin{subfigure}[htbp]{0.5\textwidth}
            \includegraphics[width=1.0\linewidth]{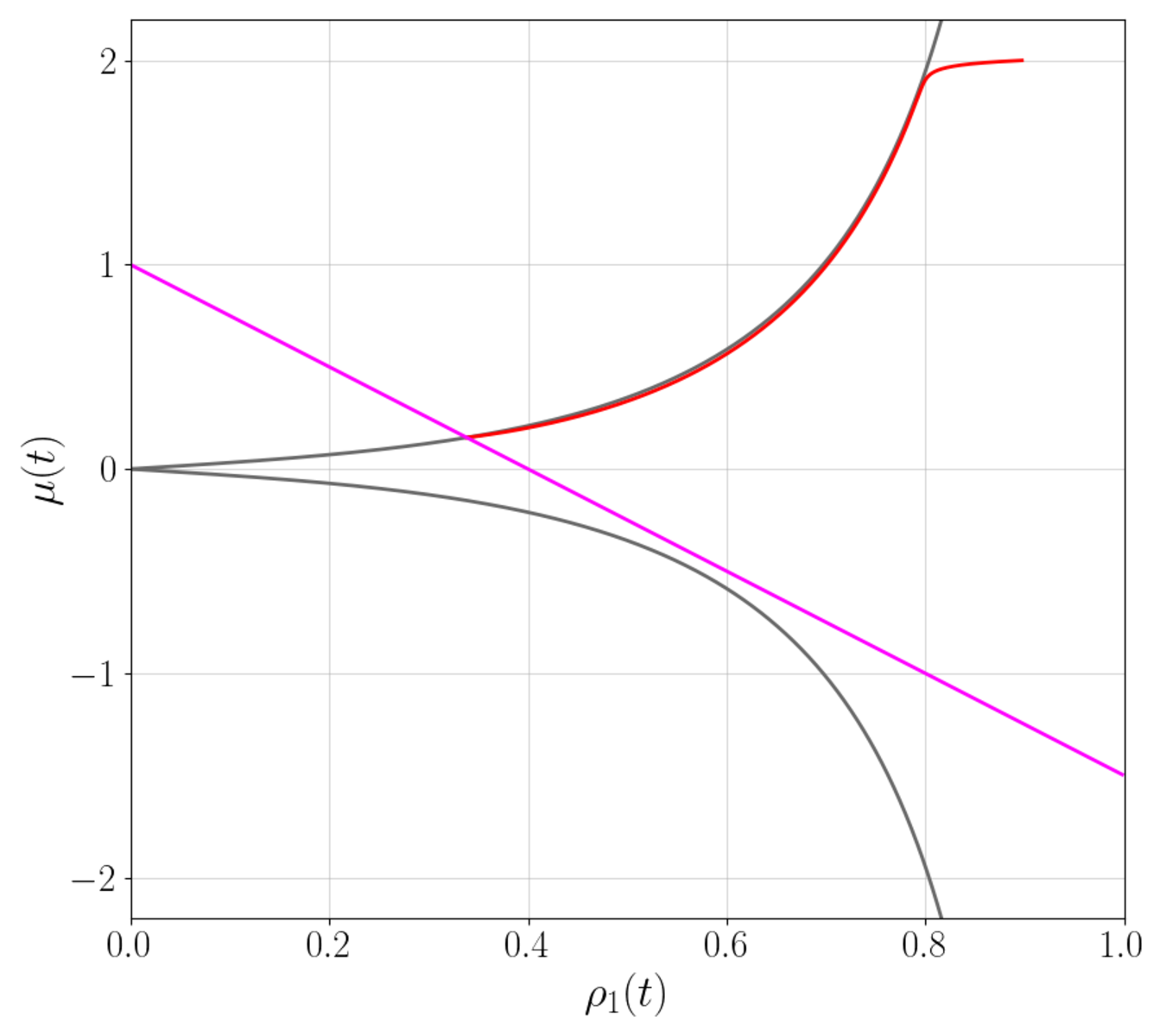}
        \end{subfigure}
        \vskip \baselineskip
        \begin{subfigure}[htbp]{0.49\textwidth}
            \includegraphics[width=1.0\linewidth]{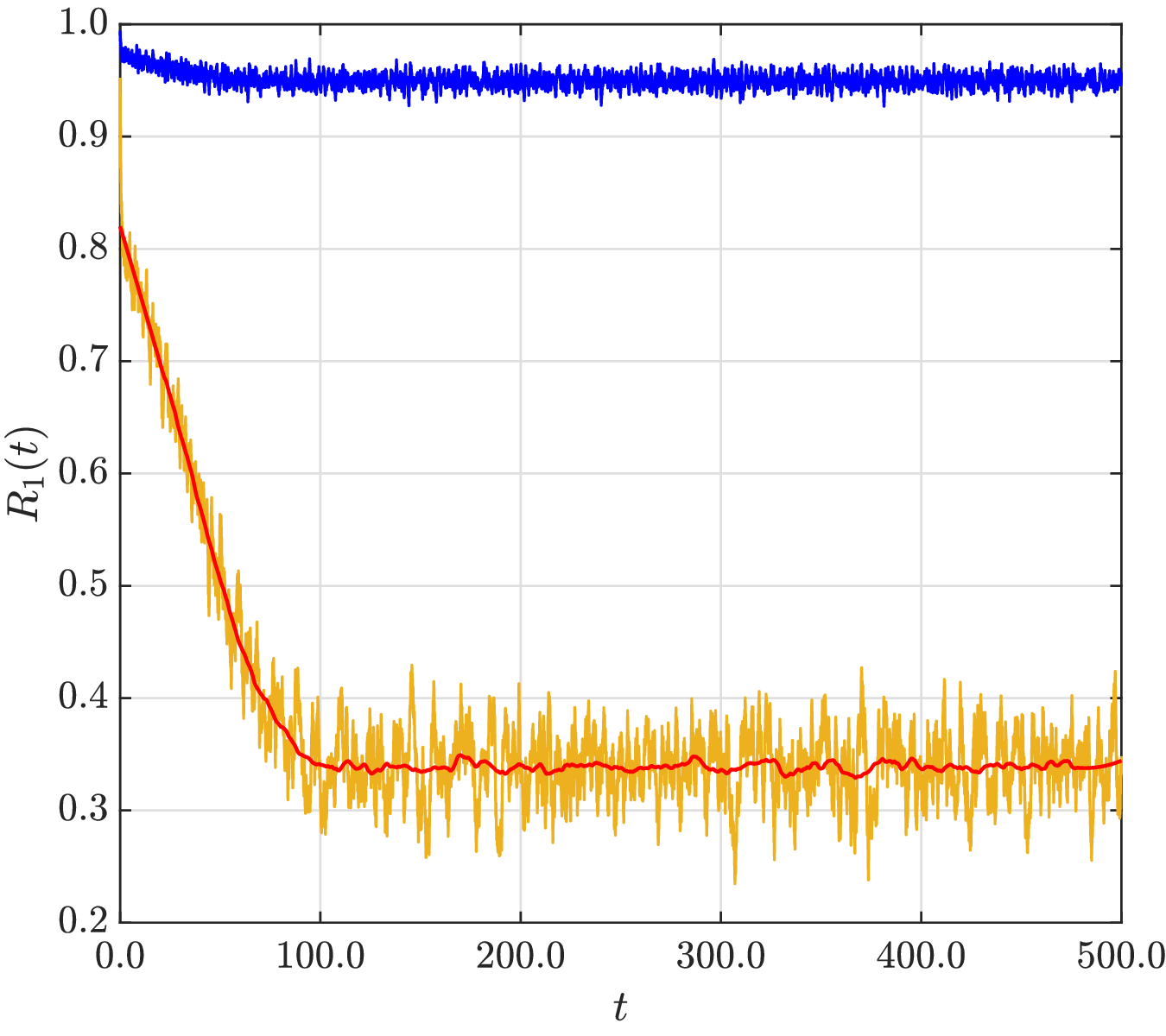}
        \end{subfigure}\hfill
        \begin{subfigure}[htbp]{0.5\textwidth}
            \includegraphics[width=1.0\linewidth]{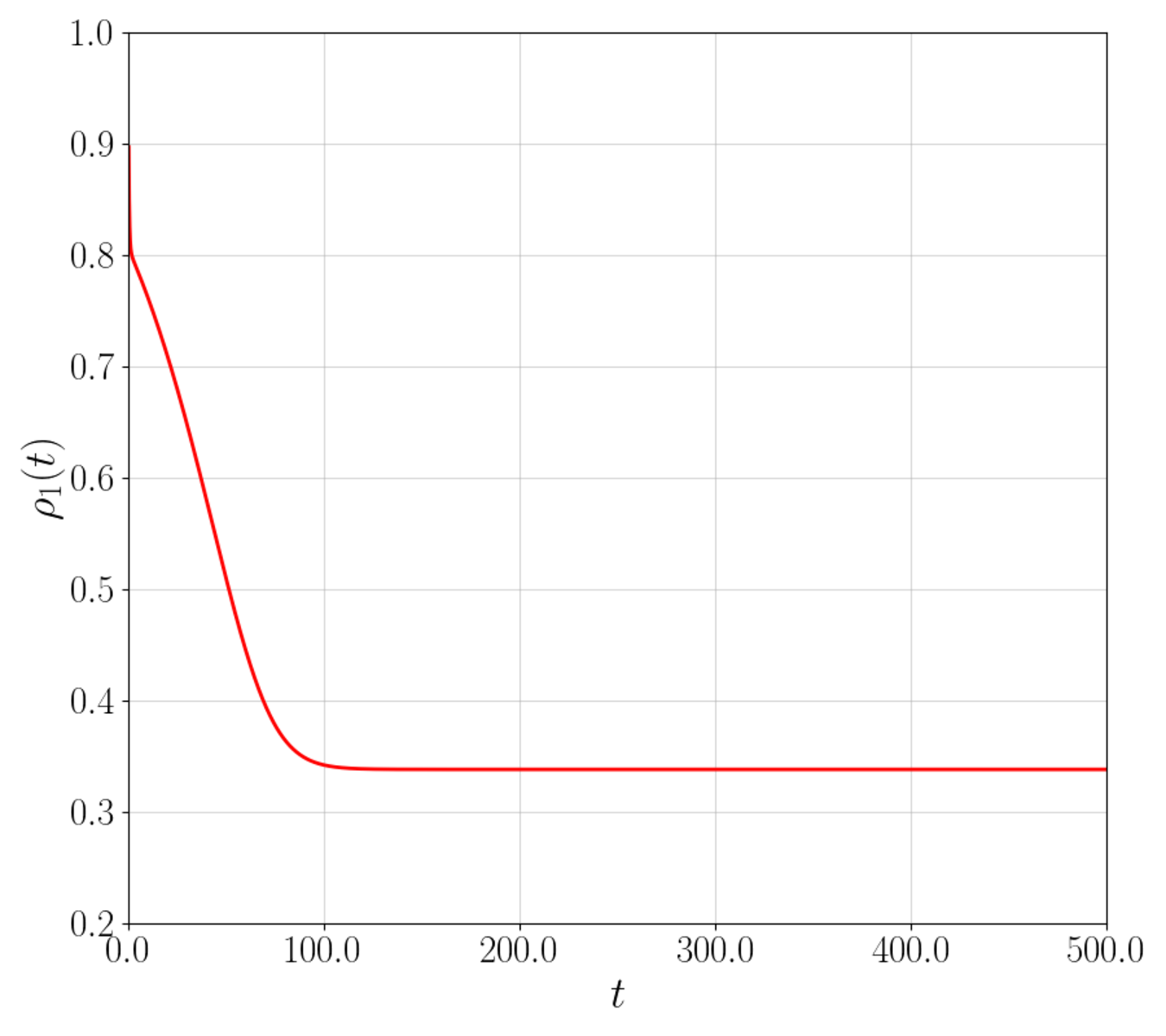}
        \end{subfigure}
        \caption{Locally exponentially stable intercoupling chimera. Upper row: network (left) and mean-field (right) projections to the $\left( R_{1},\mu \right)$ and $\left(\rho_{1},\mu\right)$ planes for \eqref{GeneralModelLGVP} and \eqref{CoevolutiveIntercouplingReducedLGVP}, respectively. Attracting critical manifold (grey), slow adaptation nullcline (magenta) and response (red). Lower row: order parameter evolution in the network (left) and mean-field (right). For the network, second (blue) and first (orange) population synchronization level and the filtered version of the latter (red). Adaptive law $\dot{\mu} = \varepsilon_{\mu}(-\mu + \gamma_{\mu} - \eta_{\mu}\rho_{1})$ and initial conditions $(\rho_{1},\psi,\mu) = (0.9, -0.5, 2.0)$. Parameters: $\Delta_{1}=0.6$, $\Delta_{2}=0.1$, $\omega_{1} = 101/20$, $\omega_{2} = \omega_{1}+0.01$, $k_{1} = 0.9$, $k_{2}=2.0$, $\gamma_{\mu} = 1.0$, $\eta_{\mu} = 2.5$, $\varepsilon_{\mu}=0.02$, \textcolor{black}{$\beta_{11}=\beta_{12}=\beta_{21}=\beta_{22}=0$}, and $N_{1,2}=300$.}
        \label{Fig:InterFeedback}
    \end{figure*}

    \begin{figure*}
        \begin{subfigure}[htbp]{0.48\textwidth}
            \centering
            \includegraphics[width=1.0\linewidth]{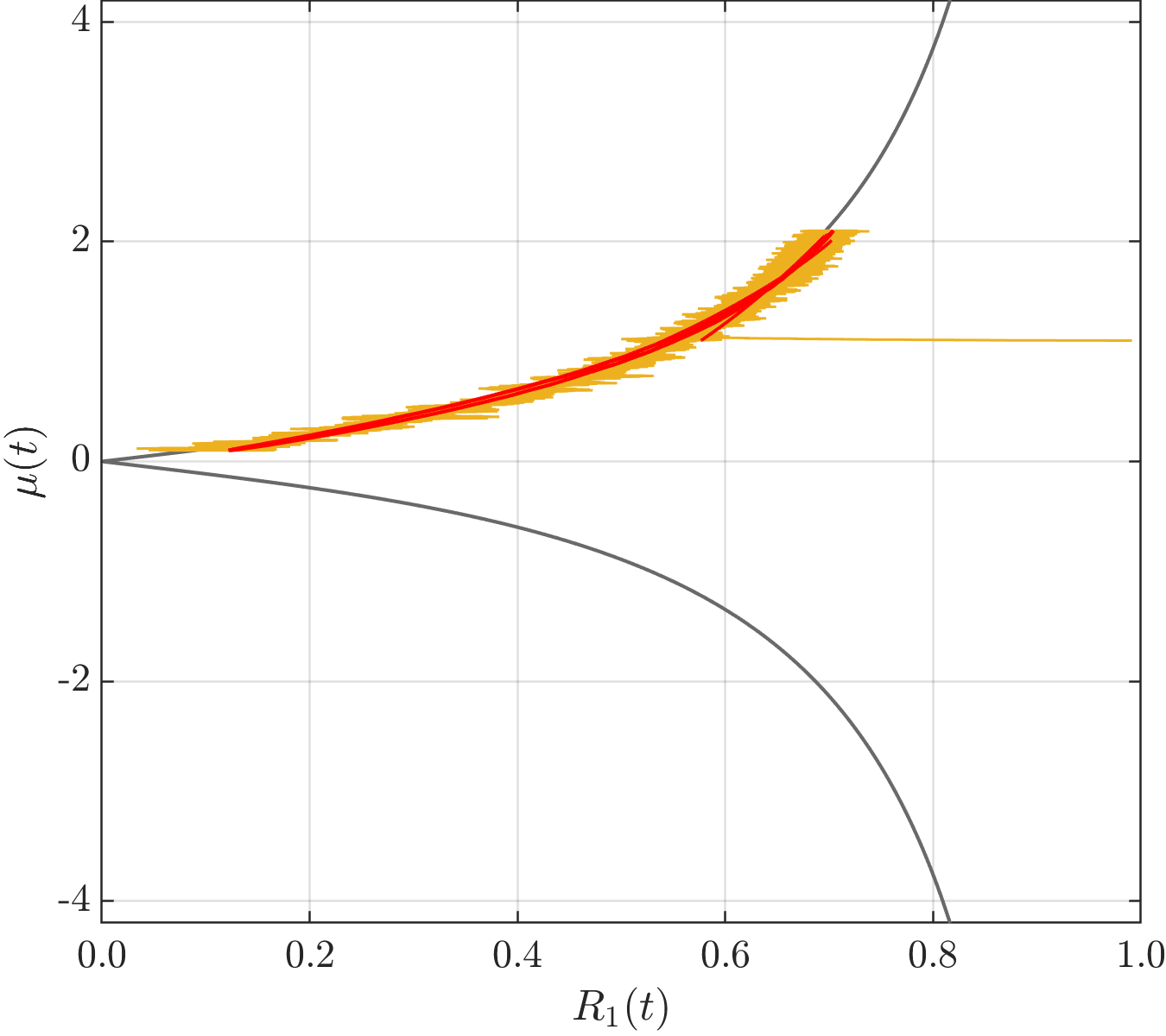}
        \end{subfigure}\hfill
        \begin{subfigure}[htbp]{0.5\textwidth}
            \centering
            \includegraphics[width=1.0\linewidth]{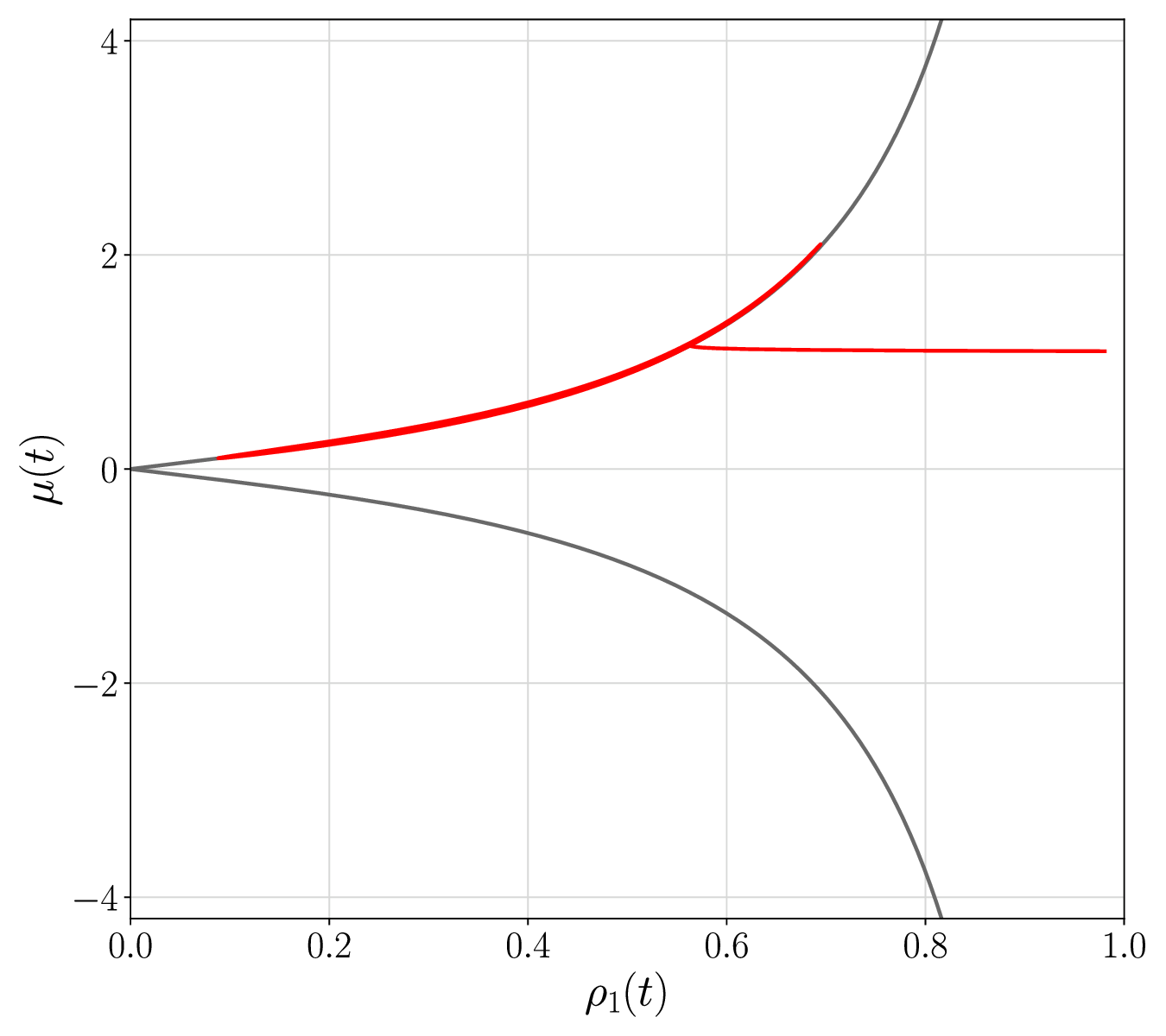}
        \end{subfigure}
        \vskip \baselineskip
        \begin{subfigure}[htbp]{0.49\textwidth}
            \centering
            \includegraphics[width=1.0\linewidth]{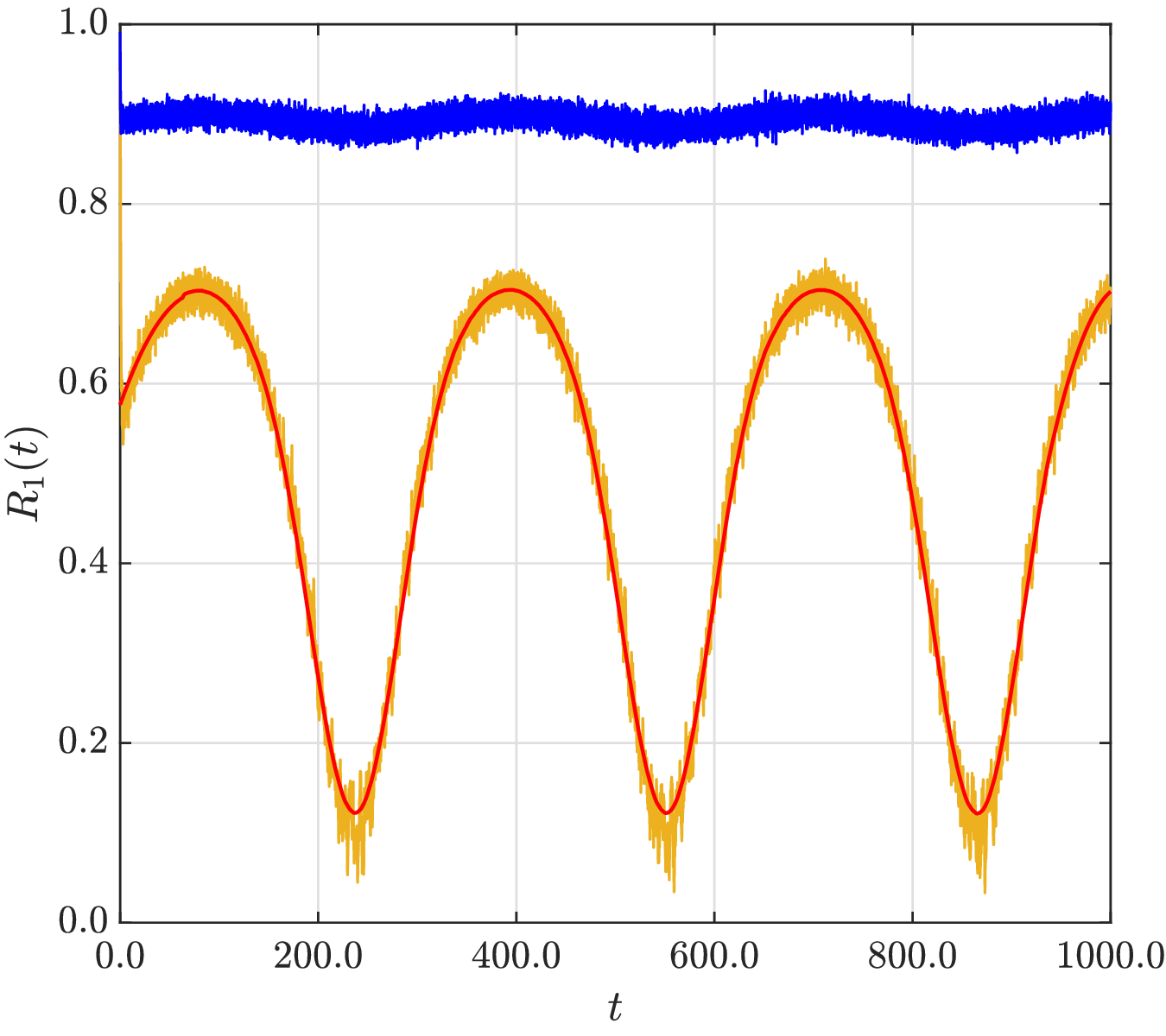}
        \end{subfigure}\hfill
        \begin{subfigure}[htbp]{0.5\textwidth}
            \centering
            \includegraphics[width=1.0\linewidth]{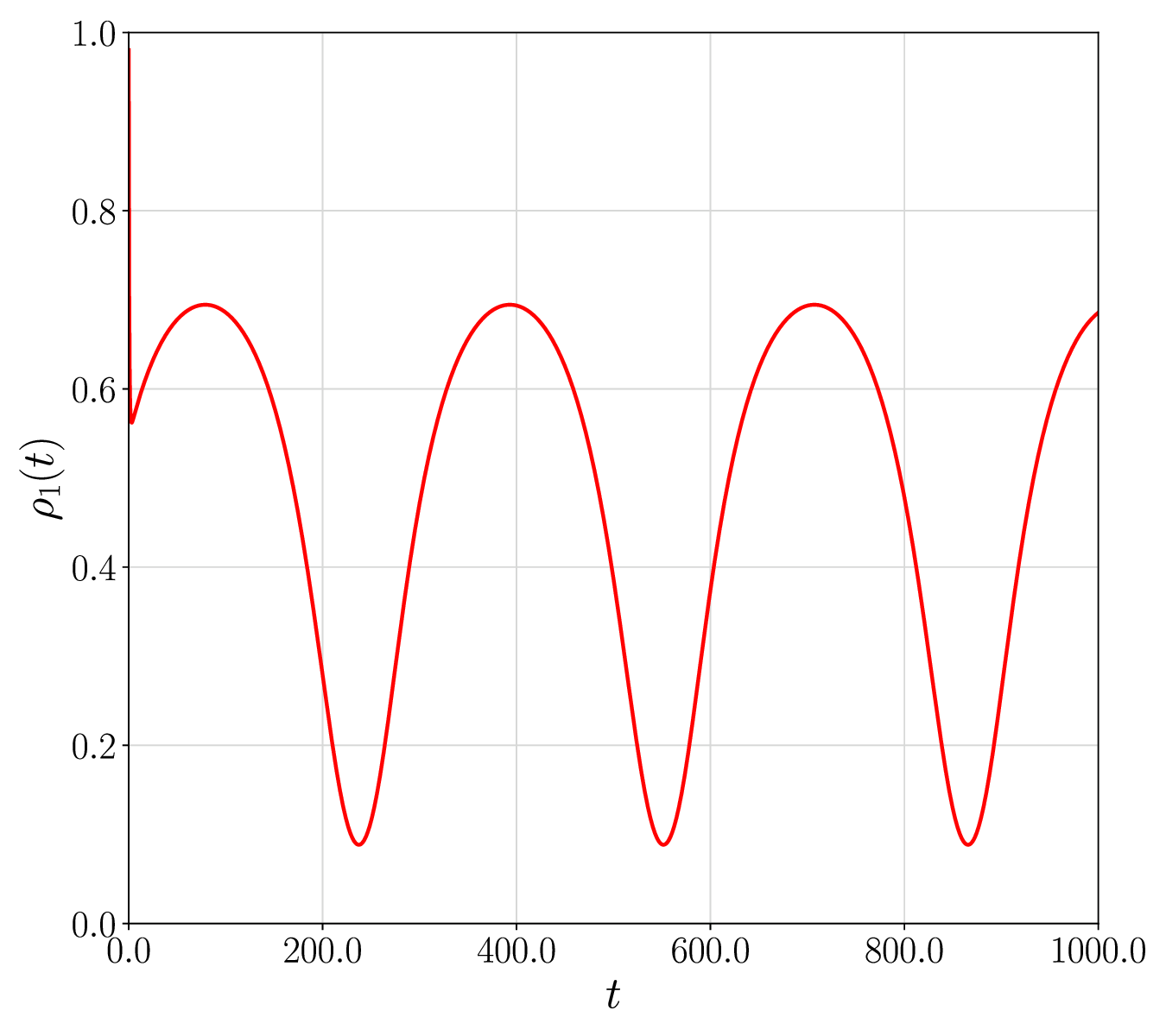}
        \end{subfigure}
        \caption{Locally exponentially stable intercoupling breathing chimera. Upper row: network (left) and mean-field (right) projections to the $\left( R_{1},\mu \right)$ and $\left(\rho_{1},\mu\right)$ planes for \eqref{GeneralModelLGVP} and \eqref{CoevolutiveIntercouplingReducedLGVP}, respectively. Attracting critical manifold (grey), and response (red). Lower row: network (left) and mean-field (right) order parameter evolution on time. For the network, second (blue) and first (orange) population synchronization level and the filtered version of the latter (red). Adaptive law $\dot{\mu} = \varepsilon_{\mu}\cos{(0.02 t)}$ and initial conditions $(\rho_{1},\psi,\mu) = (0.99, -0.5, 1.1)$. Parameters: $\Delta_{1}=1.0$, $\Delta_{2}=0.1$, $\omega_{1} = 101/20$, $\omega_{2} = \omega_{1}+0.01$, $k_{1} = 0.9$, $k_{2}=9.0$, $\varepsilon_{\mu}=0.02$, \textcolor{black}{$\beta_{11}=\beta_{12}=\beta_{21}=\beta_{22}=0$}, and $N_{1,2}=1000$.}
        \label{Fig:InterBreathing}
    \end{figure*}

\begin{subsection}{Coevolutionary Intracoupling}
Analogously to the intercoupling problem, we examine the numeric results obtained by considering system \eqref{GeneralModelLGVP} with an adaptive intracoupling strength and its associated mean-field reduction \eqref{CoevolutiveIntraReducedLGVP}. To begin with, a linear feedback rule $\dot{k}_{1}=\varepsilon_{1}(-k_{1} + \gamma_{1} - \eta_{1}\rho_{1})$ is considered and in figure \ref{Fig:IntraFeedback} we present the obtained results. As explained earlier, the critical manifold \eqref{CoevolutiveIntraCriticalManifold} of \eqref{CoevolutiveIntraReducedLGVP} has one normally hyperbolic and attracting branch, namely $\mathcal{C}_{0}^{-}$, while the other branch is at least of saddle type. In figure \ref{Fig:IntraFeedback}, the stable, saddle and unstable segments of the critical manifold \eqref{CoevolutiveIntraCriticalManifold} are represented by the continuous gray, dashed green and dotted blue lines, respectively. Observe that we have again initialized the first population of oscillators near complete synchronization and after a transient phase, it is clear that a chimera state is produced as the order parameters $\rho_{1}$ and $R_{1}$ converge to the intersection between the critical manifold and the slow flow set of equilibria. Therefore, as long as an appropriate macroscopic and slow adaptive law is designed, it is possible to produce diverse patterns due to the geometric properties of the mean-field \eqref{CoevolutiveIntraReducedLGVP}, and by means of theorem \ref{TheoremNetworkMeanField}, such long term dynamical features are reproduced in the network \eqref{GeneralModelLGVP}. Finally, it is important to highlight the chimera-like pattern produced in figure \ref{Fig:IntraFeedback}, as it is stable for a coupling beyond the critical value, i.e. $k_{1}>2\Delta_{1}$, which without any adaptation would usually lead the network to the synchronous attractor, similar to the results obtained for weakly heterogeneous oscillators \cite{Laing2009}, but now the coevolutive mechanism allows for larger heterogeneity in the network.

\begin{figure*}[htbp]
        \begin{subfigure}[htbp]{0.48\textwidth}
            \includegraphics[width=1.0\linewidth]{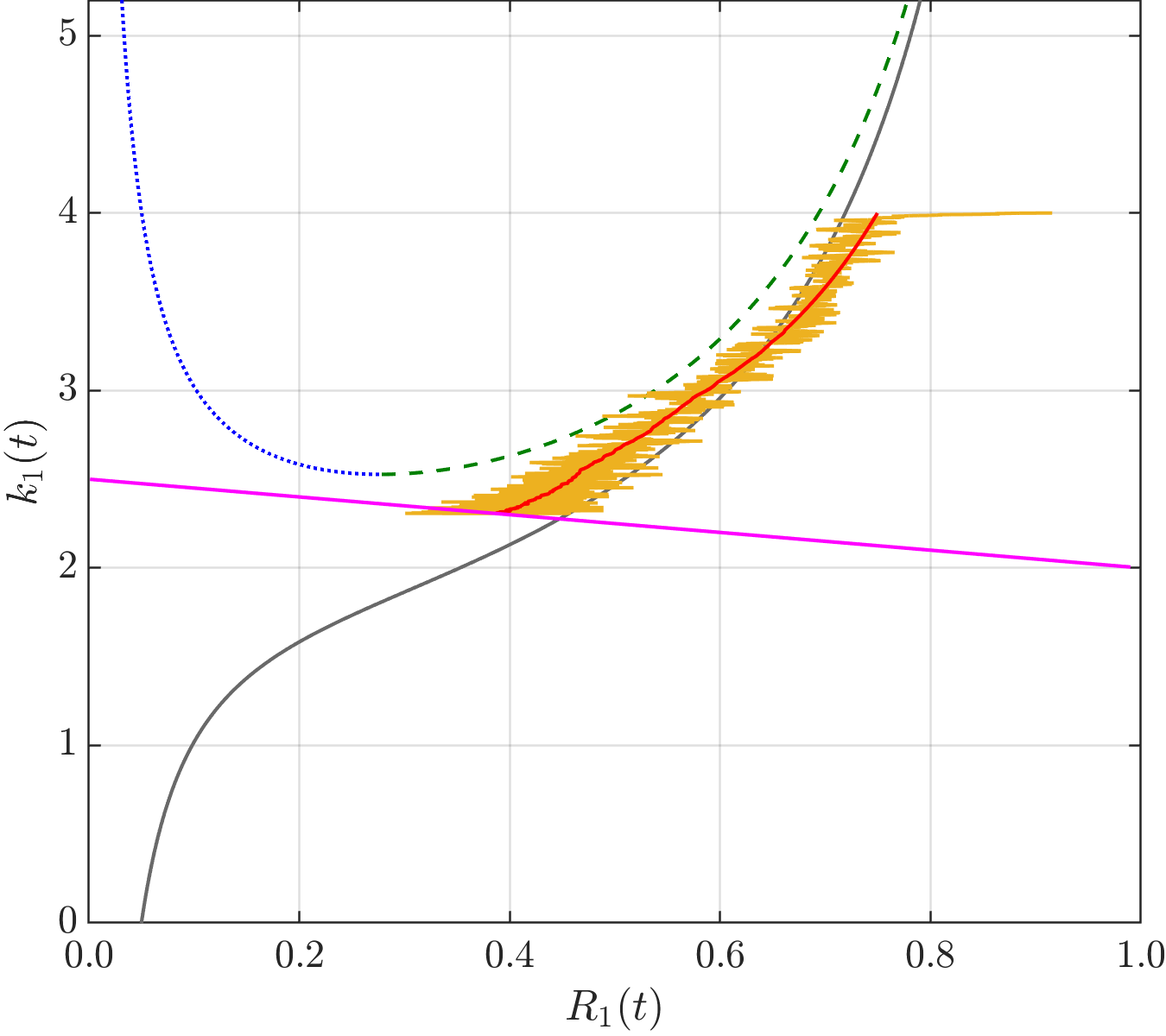}
        \end{subfigure}\hfill
        \begin{subfigure}[htbp]{0.5\textwidth}
            \includegraphics[width=1.0\linewidth]{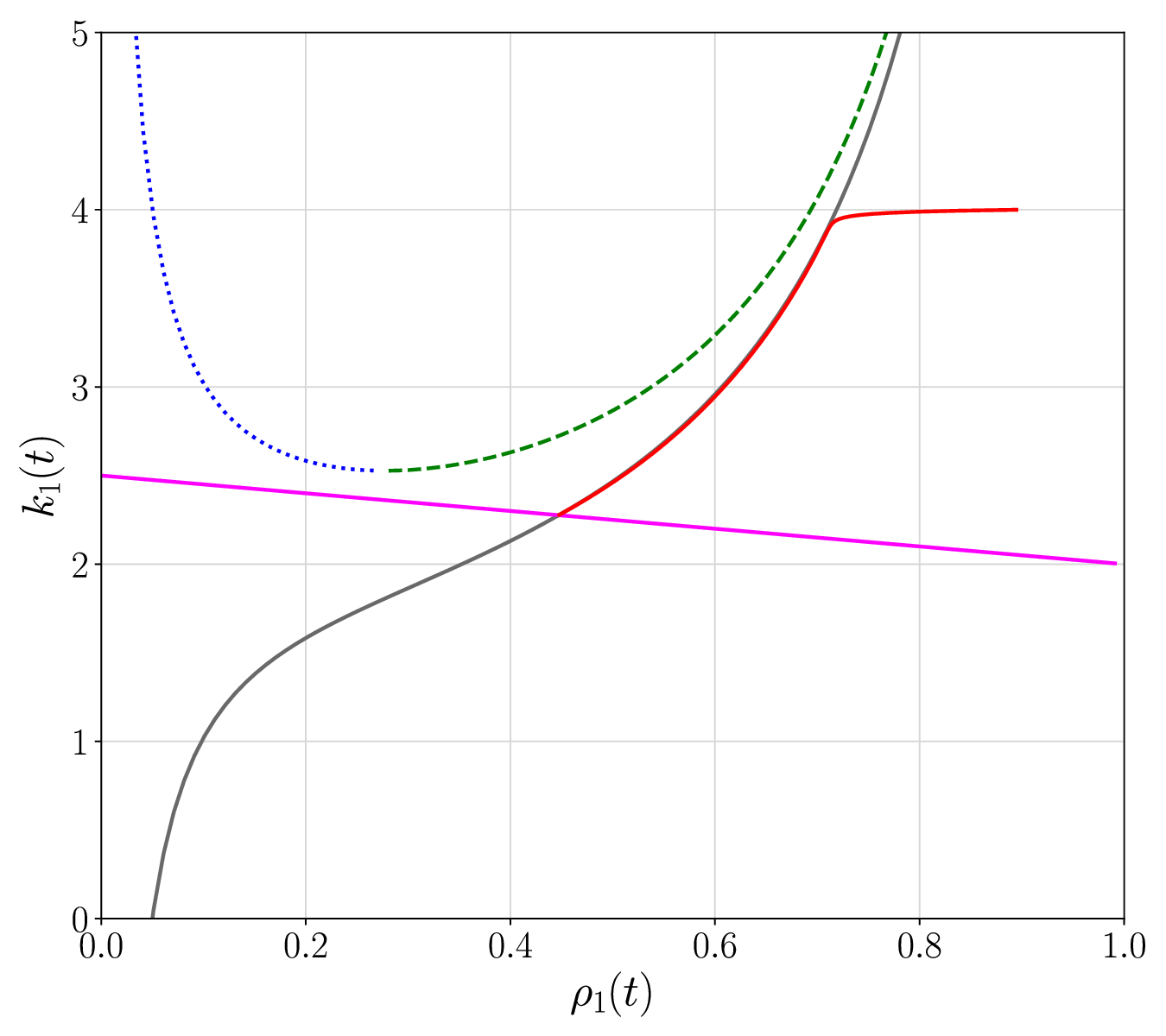}
        \end{subfigure}
        \vskip \baselineskip
        \begin{subfigure}[htbp]{0.49\textwidth}
            \includegraphics[width=1.0\linewidth]{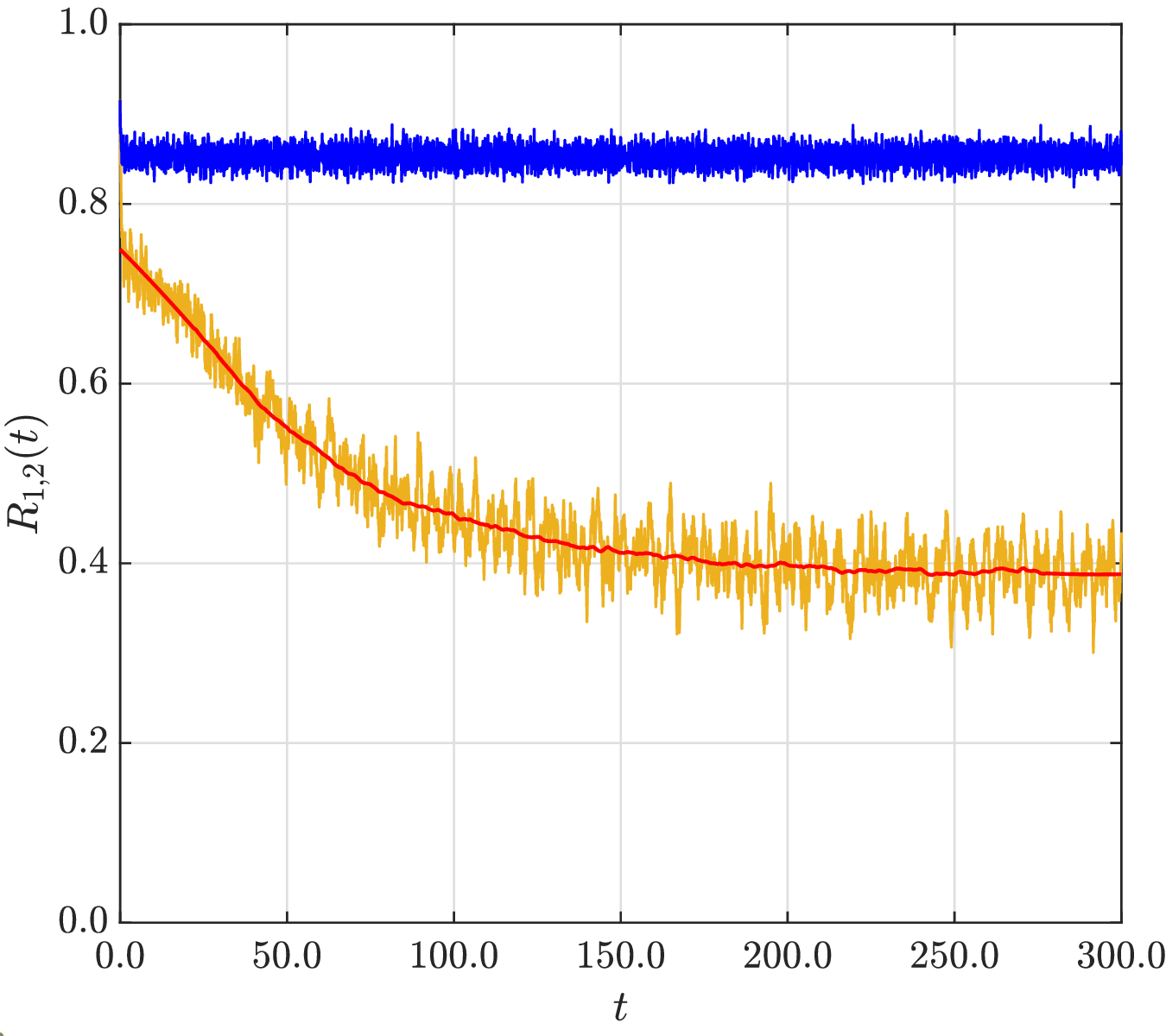}
        \end{subfigure}\hfill
        \begin{subfigure}[htbp]{0.5\textwidth}
            \includegraphics[width=1.0\linewidth]{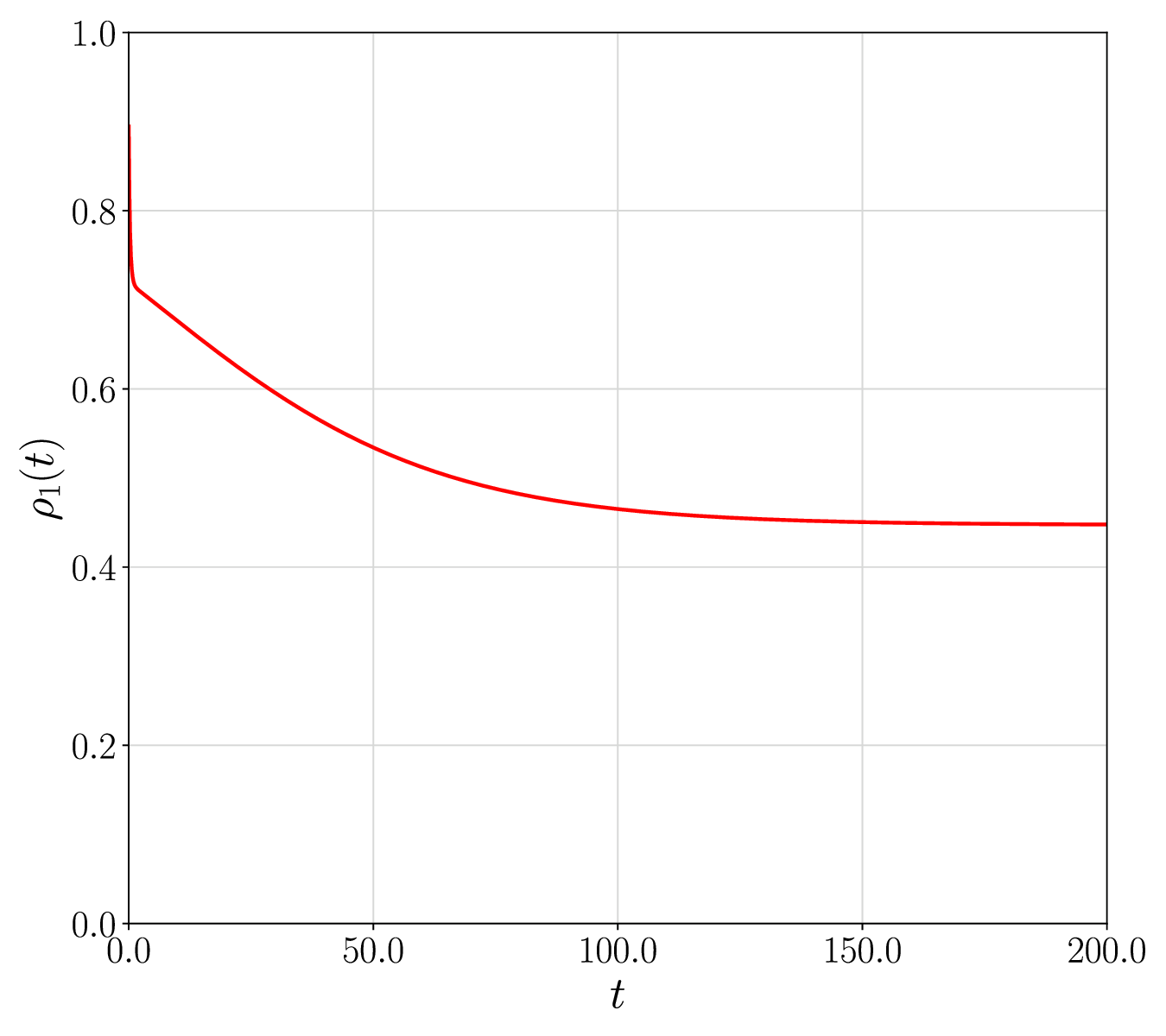}
        \end{subfigure}
        \caption{Locally exponentially stable intracoupling chimera. Upper row: network (left) and mean-field (right) projections to the $\left(R_{1},k_{1}\right)$ and $\left(\rho_{1},k_{1}\right)$ planes for \eqref{GeneralModelLGVP} and \eqref{CoevolutiveIntraReducedLGVP}, respectively. Critical manifold regions: attracting (grey), saddle (dashed green), repelling (blue dotted), slow adaptation nullcline (magenta) and response (red).  Lower row: order parameter evolution in the network (left) and mean-field (right). For the network, second (blue) and first (orange) population synchronization level and the filtered version of the latter (red). Adaptive law $\dot{k}_{1} = \varepsilon_{1}(-k_{1} + \gamma_{1} - \eta_{1}\rho_{1})$ and initial conditions $(\rho_{1},\psi,k_{1}) = (0.9, -0.5, 4.0)$. Parameters: $\Delta_{1}=1.0$, $\Delta_{2}=0.01$, $\omega_{1} = 101/20$, $\omega_{2} = \omega_{1}$, $\mu = 0.1$, $k_{2} = 4.0$, $\gamma_{1} = 2.5$, $\eta_{1} = 0.5$,  $\varepsilon_{1}=0.02$, \textcolor{black}{$\beta_{11}=\beta_{12}=\beta_{21}=\beta_{22}=0$}, and $N_{1,2}=500$.}
        \label{Fig:IntraFeedback}
    \end{figure*}
\end{subsection}

\begin{remark}
\label{Rem:PhaseLag}
    Despite the well-known importance of the phase-lag parameter in determining the stability of the chimera state \cite{Nakagawa1993, Nakagawa1994, Kuramoto2002, Bolotov2018}, the assumption $\beta_{\sigma \sigma'} = 0$ is done in order to simplify the analysis, as the derived results hold for $\beta_{\sigma \sigma'}$ greater than zero and sufficiently small. In this regard, following the results for the Ott-Antonsen ansatz applied to an homogeneous two-population network with the same intracouplings but different intercoupling between layers \cite{Abrams2008}, the resulting mean-field equations represent our reduced system in the limit $\beta_{\sigma \sigma'}\xrightarrow{}0$. The aforementioned fact is due to the geometric properties of the system and perturbation results, which is why it is possible to produce different synchronization patterns, including chimera states, even when in the presence of the phase-lag. For instance, in figure \ref{Fig:PhaseLag}, we present an example of a chimera state in a coevolutionary intercoupling scenario with a phase-lag $\beta_{11} = \beta_{12} = \beta_{21} = \beta_{22} = \pi/6$. \textcolor{black}{Similarly, in figure \ref{Fig:PhaseLagIntra}, we show a stable chimera state in the coevolutionary intracoupling case for which $\beta_{11} = \beta_{12} = \beta_{21} = \beta_{22} = \pi/10$.} Naturally, the trajectory follows a manifold $\mathcal{C}_{\varepsilon}$ sufficiently close to the critical manifold $\mathcal{C}_{0}^{\pm}$ in grey, until it reaches the crossing between $\mathcal{C}_{\varepsilon}$ and the slow adaptation nullcline. \textcolor{black}{Notice that the apparent shift between $\mathcal{C}_{\varepsilon}$ and $\mathcal{C}_{0}^{\pm}$ is due to the phase-lag considered but, as long as such parameter is sufficiently small, the normally hyperbolic structure is preserved.} Therefore, the behavior obtained is in the end a stable chimera state, as can be appreciated from the local order parameters of the aforementioned examples.
\end{remark}

\begin{figure*}[htbp]
        \begin{subfigure}[htbp]{0.49\textwidth}
            \includegraphics[width=1.0\linewidth]{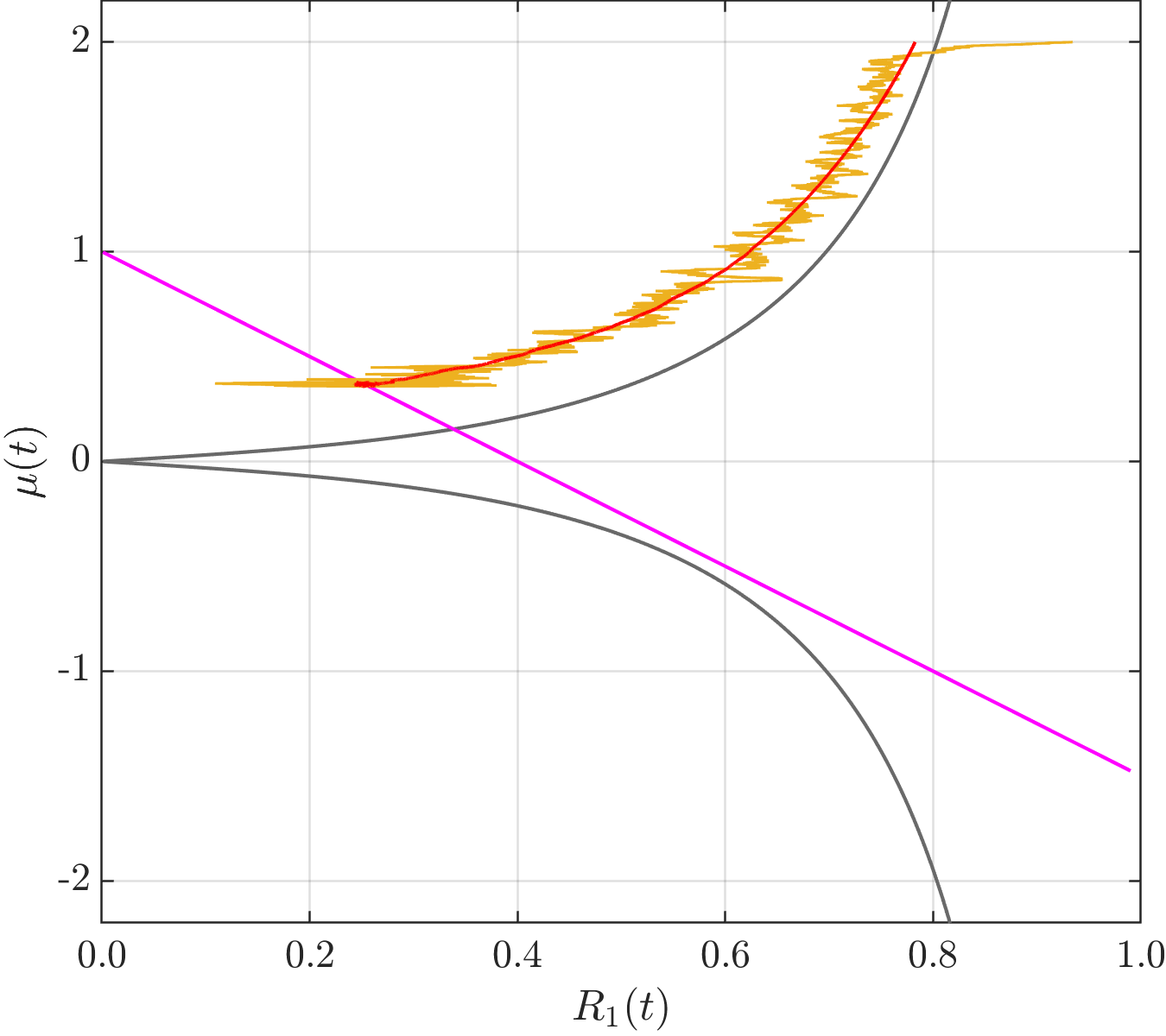}
        \end{subfigure}\hfill
        \begin{subfigure}[htbp]{0.5\textwidth}
            \includegraphics[width=1.0\linewidth]{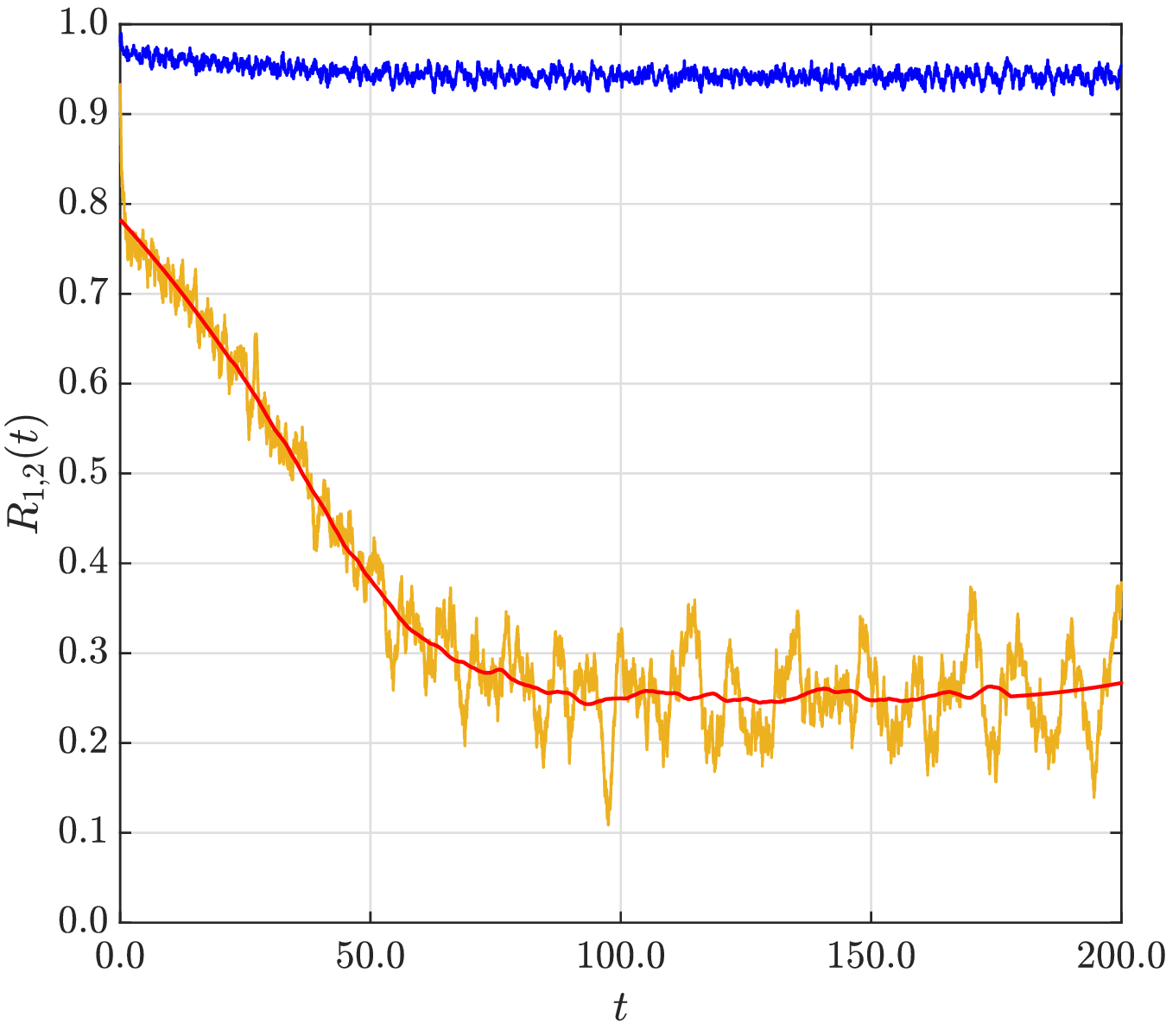}
        \end{subfigure}
        \caption{Generation of a stable chimera state on a network with phase lag \textcolor{black}{and a coevolutive intercoupling strength}. Projection to the $\left( R_{1},\mu \right)$ plane (left) and local order parameters (right) for the coevolutive intercoupling case, respectively. Attracting critical manifold (grey), slow adaptation nullcline (magenta), first population synchronization level (orange) and its filtered version (red), second population order parameter (blue). Adaptive law $\dot{\mu} = \varepsilon_{\mu}(-\mu + \gamma_{\mu} - \eta_{\mu}\rho_{1})$ and initial conditions $(\rho_{1},\psi,\mu) = (0.9, -0.5, 2.0)$. Parameters: $\Delta_{1}=0.6$, $\Delta_{2}=0.1$, $\omega_{1} = 101/20$, $\omega_{2} = \omega_{1} + 0.01$, $k_{1} = 0.9$, $k_{2}=2.0$, $\gamma_{\mu} = 1.0$, $\eta_{\mu} = 2.5$, $\varepsilon_{\mu}=0.02$, \textcolor{black}{$\beta_{11}=\beta_{12}=\beta_{21}=\beta_{22} = \pi/6$}, and $N_{1,2}=500$.}
        \label{Fig:PhaseLag}
    \end{figure*}

\begin{figure*}[htbp]
        \begin{subfigure}[htbp]{0.49\textwidth}
            \includegraphics[width=1.0\linewidth]{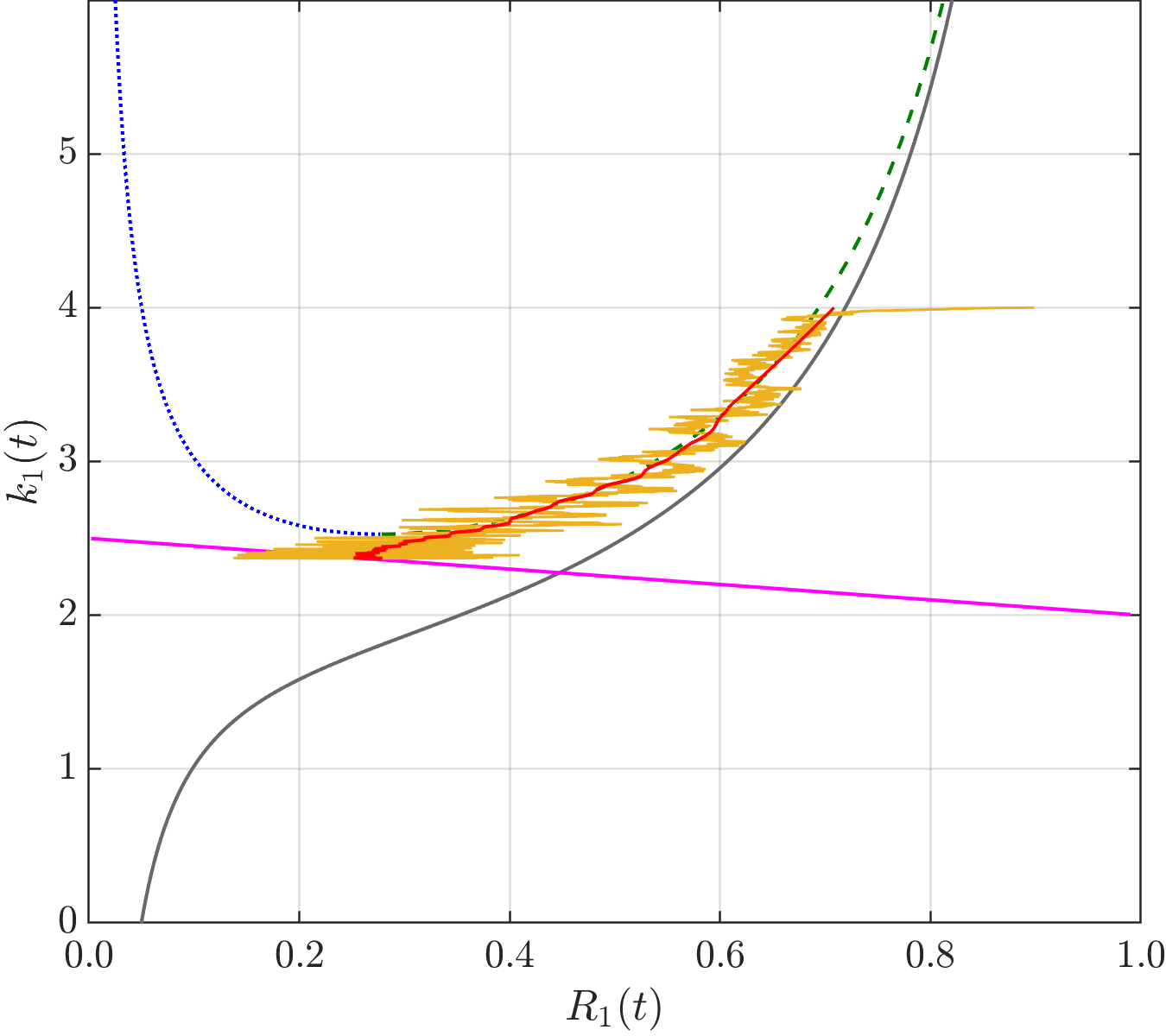}
        \end{subfigure}\hfill
        \begin{subfigure}[htbp]{0.5\textwidth}
            \includegraphics[width=1.0\linewidth]{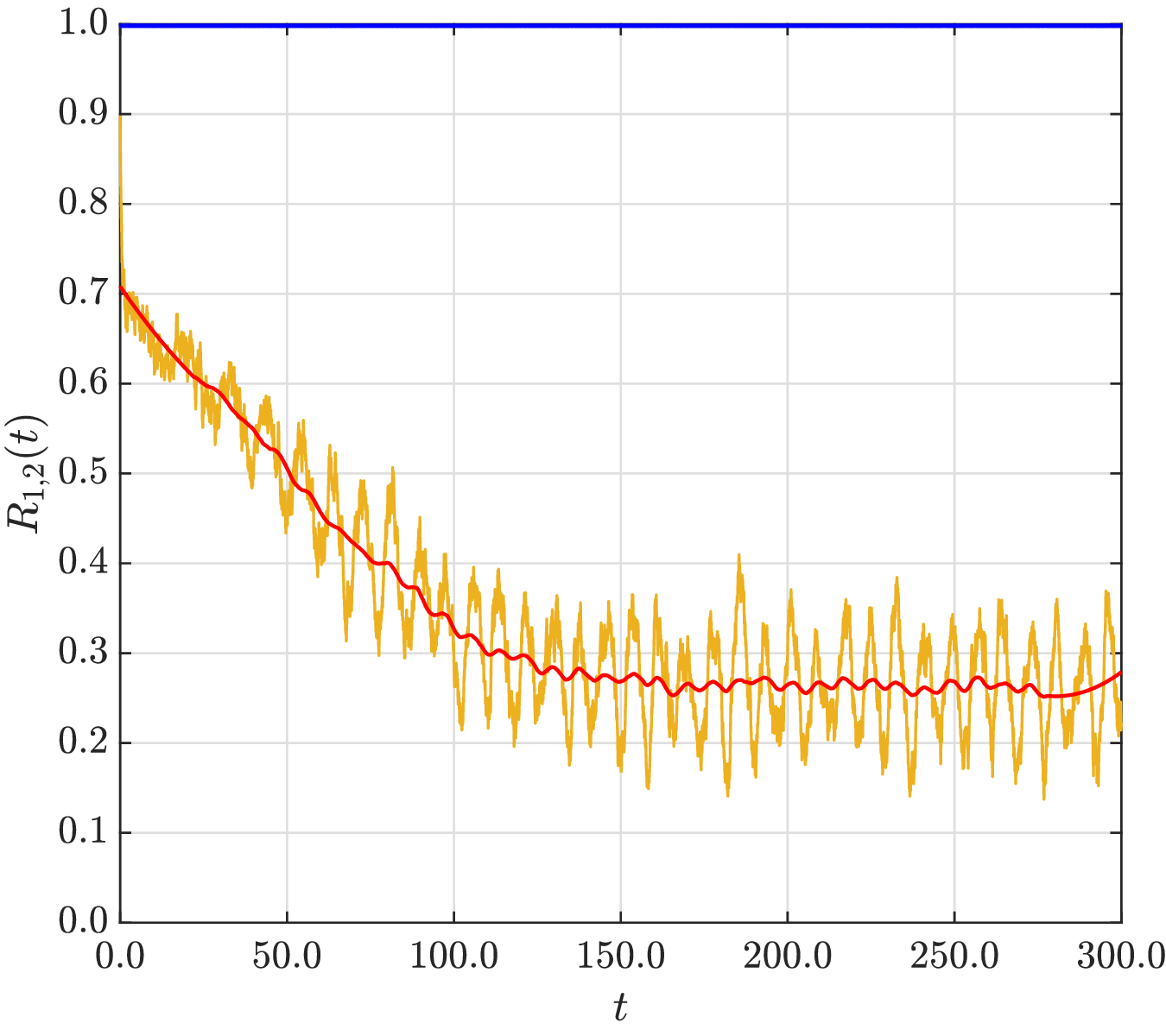}
        \end{subfigure}
        \caption{\textcolor{black}{Generation of a stable chimera state on a network with phase lag and a coevolutive intracoupling strength. Projection to the $\left( R_{1},k_{1} \right)$ plane (left) and local order parameters (right) for the coevolutive intracoupling case, respectively. Critical manifold regions: attracting (grey), saddle (dashed green), repelling (blue), slow adaptation nullcline (magenta), first population synchronization level (orange) and its filtered version (red), second population order parameter (blue). Adaptive law $\dot{k}_{1} = \varepsilon_{1}(-k_{1} + \gamma_{1} - \eta_{1}\rho_{1})$ and initial conditions $(\rho_{1},\psi,k_{1}) = (0.9, -0.5, 4.0)$. Parameters: $\Delta_{1}=1.0$, $\Delta_{2}=0.01$, $\omega_{1} = 101/20$, $\omega_{2} = \omega_{1}$, $\mu = 0.1$, $k_{2}=4.0$, $\gamma_{1} = 2.5$, $\eta_{1} = 0.5$, $\varepsilon_{1}=0.02$, $\beta_{11}=\beta_{12}=\beta_{21}=\beta_{22} = \pi/10$, and $N_{1,2}=1000$.}}
        \label{Fig:PhaseLagIntra}
    \end{figure*}

\begin{remark}
    It is important to emphasize that, although there is not an evident and spontaneous symmetry breaking, different synchronization patterns, including the chimera-like behaviors, can be generated for both the inter and intracoupling scenarios in the network \eqref{GeneralModelLGVP} and the mean-field \eqref{FullReducedSystemLGVP} by the inclusion of a proper macroscopic adaptive law with an adequate set of parameters. For instance, it is possible to produce a highly synchronized network even when the given parameters, without adaptation, would produce an incoherent system, as depicted in figure \ref{Fig:IntercouplingSync} for the coevolutionary intercoupling case. Therefore, the inclusion of the adaptation in a network remarkably expands the capability of the previously known results of such systems. Finally, we would like to highlight that setting the initial condition of the level of synchrony $R_{1}(t)$ in the network is a difficult task as it is determined by the average of the pseudo-randomly generated initial phases of the oscillators, which is why the initial conditions of $R_{1}(t)$ and $\rho_{1}(t)$ are slightly different in figure \ref{Fig:IntercouplingSync}.
\end{remark}

\begin{figure*}[htbp]
        \begin{subfigure}[htbp]{0.49\textwidth}
            \includegraphics[width=1.0\linewidth]{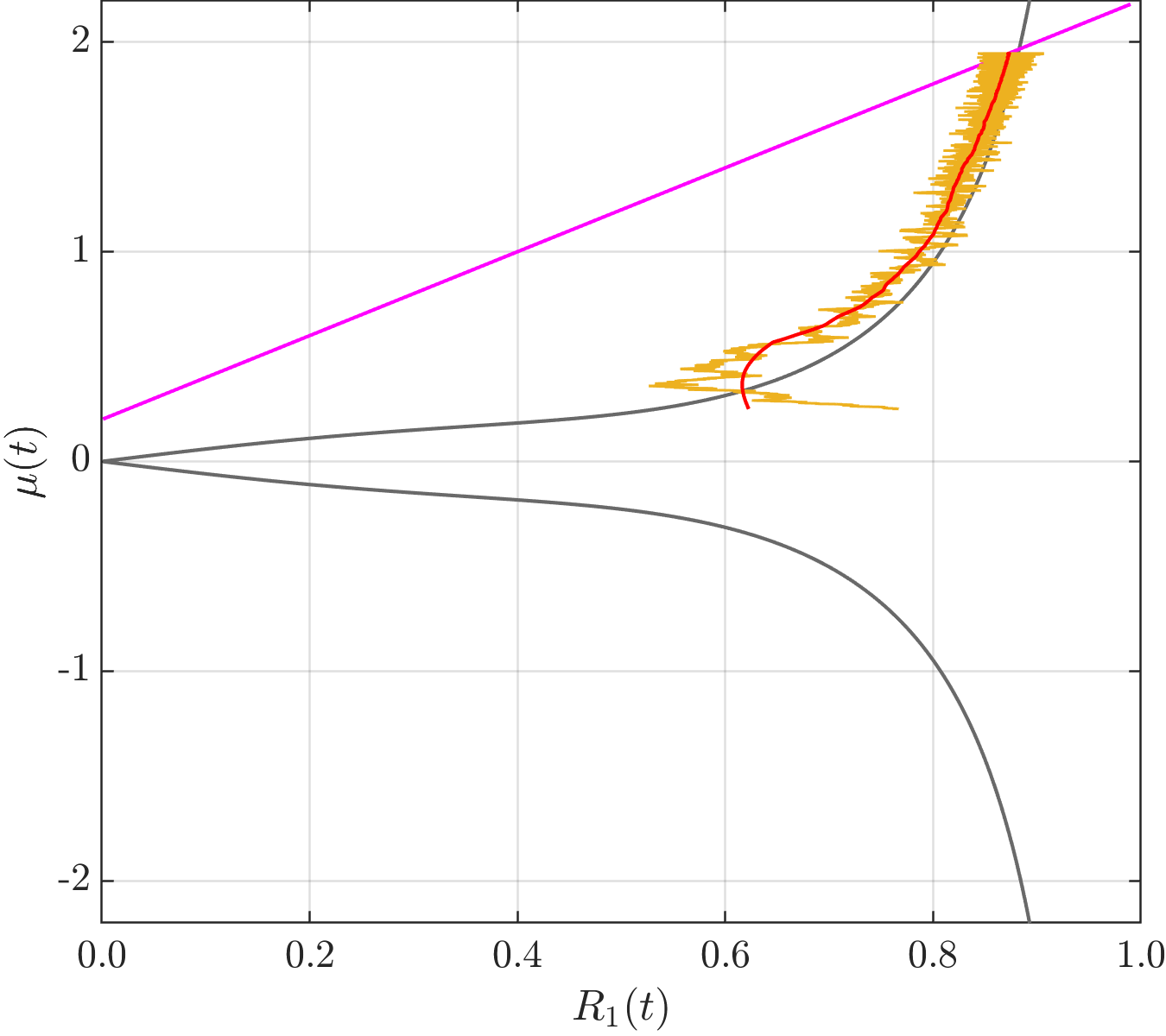}
        \end{subfigure}\hfill
        \begin{subfigure}[htbp]{0.5\textwidth}
            \includegraphics[width=1.0\linewidth]{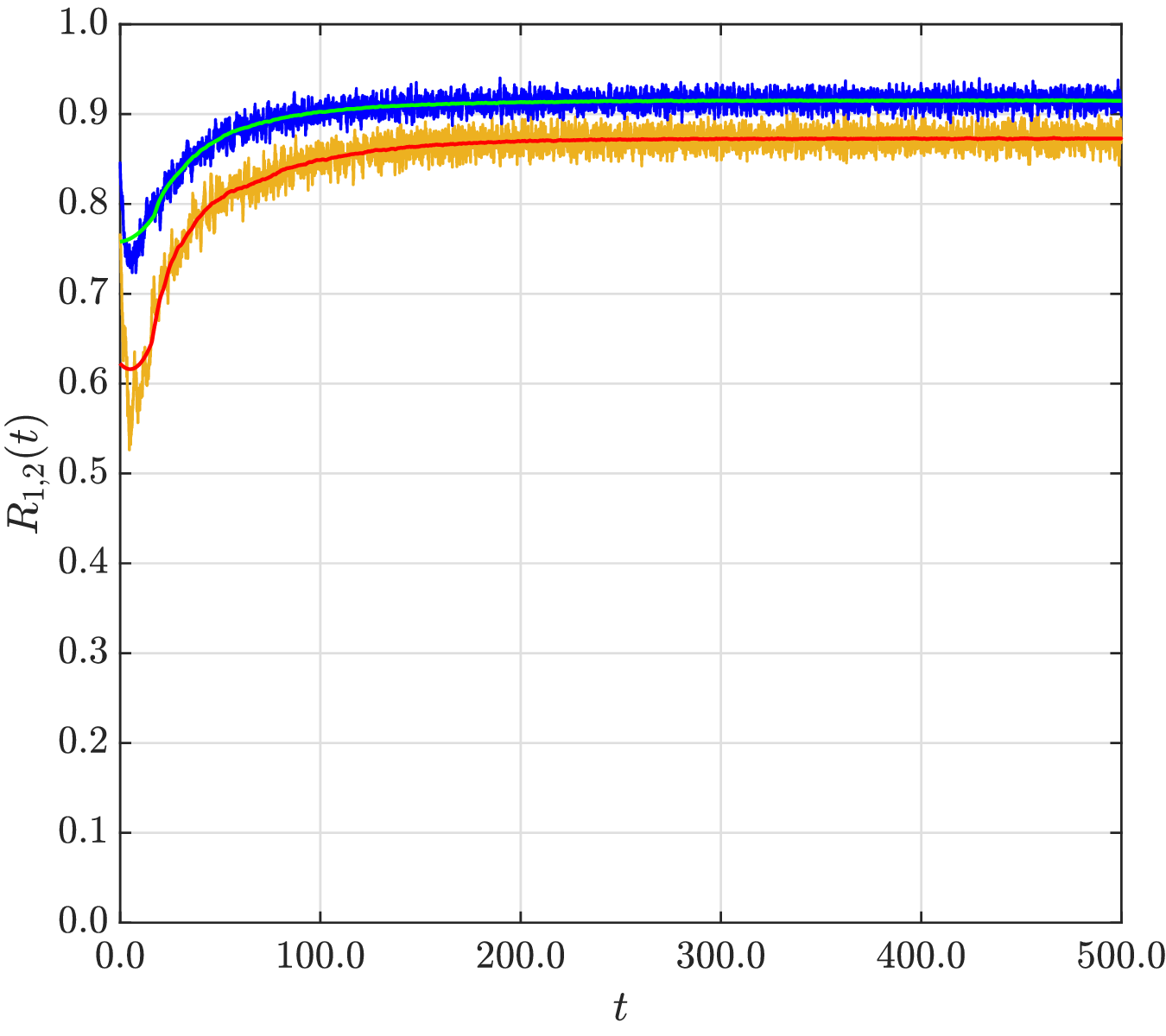}
        \end{subfigure}
        \caption{Synchronization of a partially incoherent system. Projection to the $\left( R_{1},\mu \right)$ plane and local order parameters $R_{1,2}(t)$ (right), respectively. Attracting critical manifold (grey), slow adaptation nullcline (magenta) and response (red). First and second population synchronization level (orange, blue) and their filtered version (red, green). Adaptive law $\dot{\mu} = \varepsilon_{\mu}(-\mu + \gamma_{\mu} - \eta_{\mu}\rho_{1})$ and initial conditions $(\rho_{1},\psi,\mu) = (0.4, 0.4, 0.25)$. Parameters: $\Delta_{1}=0.3$, $\Delta_{2}=0.2$, $\omega_{1} = 0.0$, $\omega_{2} = 0.3$, $k_{1} = 0.5$, $k_{2}=0.7$, $\gamma_{\mu} = 0.2$, $\eta_{\mu} = -2.0$, $\varepsilon_{\mu}=0.02$, \textcolor{black}{$\beta_{11}=\beta_{12}=\beta_{21}=\beta_{22}=0$}, and $N_{1,2}=500$. Notice that, since $k_1<2\Delta_{1}$, the observed partial synchronization is due to the adaptation, as otherwise the first population would tend to an incoherent state.}
        \label{Fig:IntercouplingSync}
    \end{figure*}

\section{Discussion and conclusions}
\label{Sec:Conclusions}
    The term \textit{chimera state} gather a broad collection of highly interesting unexpected synchronization patterns. In this work, we have shown a new mechanism for the generation of stable chimera states in a complete multilayer network of heterogeneous Kuramoto oscillators through the inclusion of slow coevolutive coupling strenghts. Particularly, we use the Ott-Antonsen ansatz to derive a mean-field representation of the general network, allowing the use of GSPT results to obtain a deeper insight of the dynamics occurring in the system. At first, we have analyzed a two-population network with only one coevolving intercoupling strength, and determined the necessary and sufficient conditions for the critical manifold to be normally hyperbolic and attracting. For such a scenario, we have effectively generated stable chimera states as well as persistent breathing chimera states only by adequately modifying the macroscopic adaptive law dictating the slow coevolving dynamics. Similarly, we studied the case of the same network with one coevolutive intracoupling strength and obtained relevant geometric results that enable the production of stable chimera states, as well as different synchronization patterns, by appropriately modifying the adaptation law employed. \textcolor{black}{Additionally, we have shown the effectiveness of our mechanism for the generation of diverse synchronization patterns in a network of coupled phase oscillators with coevolutive coupling strengths when there exists a phase-lag between the elements of the ensemble. The aforementioned fact holds as long as such parameter is sufficiently small, as it can be regarded as a perturbation in geometric terms for the given problem.} Therefore, the presented numeric results validate our mechanism for the generation of diverse patterns in a complete network of Kuramoto phase oscillators with coevolving macroscopic coupling strengths.
    \paragraph*{}Considering the results we have obtained, it is natural and intriguing to wonder about which kind of behaviours are observed in the non-normally hyperbolic scenario for both types of adaptive couplings. Although rigorous mathematical analysis is currently under development, numerical results have shown remarkable behaviors. \textcolor{black}{Firstly, it is possible to produce oscillating synchronization patterns without a periodically forced adaptive law. For instance, in figure \ref{Fig:BreathingTwoAdaptations} a stable breathing chimera is presented for a two-population system with two-coevolving coupling strengths, $\mu(t)$ and $k_{1}(t)$. As it can be seen, the periodic oscillation is sustained for the given set of parameters and adaptive laws employed. On the other hand, we suspect that canard cycles and breathing chimeras share a connection since they present certain qualitative resemblance. Following this idea, in figures \ref{Fig:InterCanardHopf} and \ref{Fig:TranscriticalCanard} two canard cycles for the two-population mean field with coevolving intercoupling scenario are shown. The technicalities and rigorous analysis are part of future work, but here we want to present some numerical evidence. In figure \ref{Fig:InterCanardHopf}, a canard raising from a singular Hopf bifurcation is depicted. For the selected parameters, the critical manifold folds at a Hopf bifurcation from where the trajectories follow the saddle branch of the critical manifold, giving rise to the depicted canard cycle. Similarly, in figure \ref{Fig:TranscriticalCanard}, a canard cycle with an oscillatory pattern due to a transcritical bifurcation is shown. Hence, the synchronization level in the first population oscillates periodically, which is why we suspect there exists a connection between canard cycles and breathing chimeras. Nevertheless, these dynamical behaviours are extremely sensitive to initial conditions and parameter variations, rendering highly complex the task of observing them in the network. For this reason, here we show only the results obtained for the mean-field and prefer to present the network results once the non-normally hyperbolic cases are better comprehended.} Moreover, we emphasize that the mean-field is only an approximation of the full system and some transient dynamics are not completely captured by this reduction. Thus, a closer correspondence between the network \eqref{GeneralModelLGVP} and the mean-field \eqref{CoevolutiveIntercouplingReducedLGVP} dynamics is achieved by increasing the population size with an adequate level of heterogeneity in the network, either by means of the natural frequency distributions or a phase-lag between the oscillators. 
    \begin{figure*}[htbp]
        \begin{subfigure}[htbp]{0.49\textwidth}
            \includegraphics[width=1.0\linewidth]{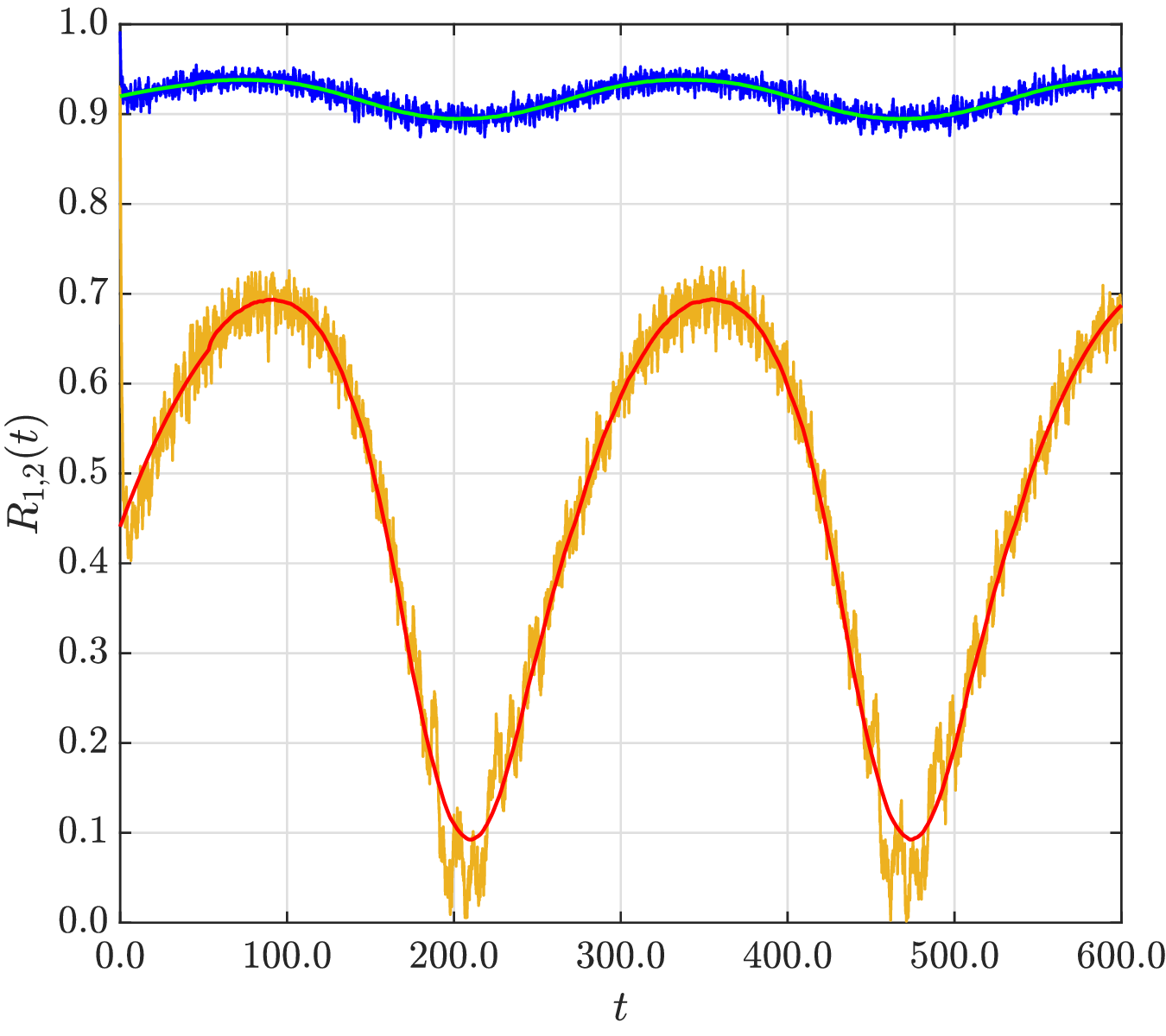}
        \end{subfigure}\hfill
        \begin{subfigure}[htbp]{0.5\textwidth}
            \includegraphics[width=1.0\linewidth]{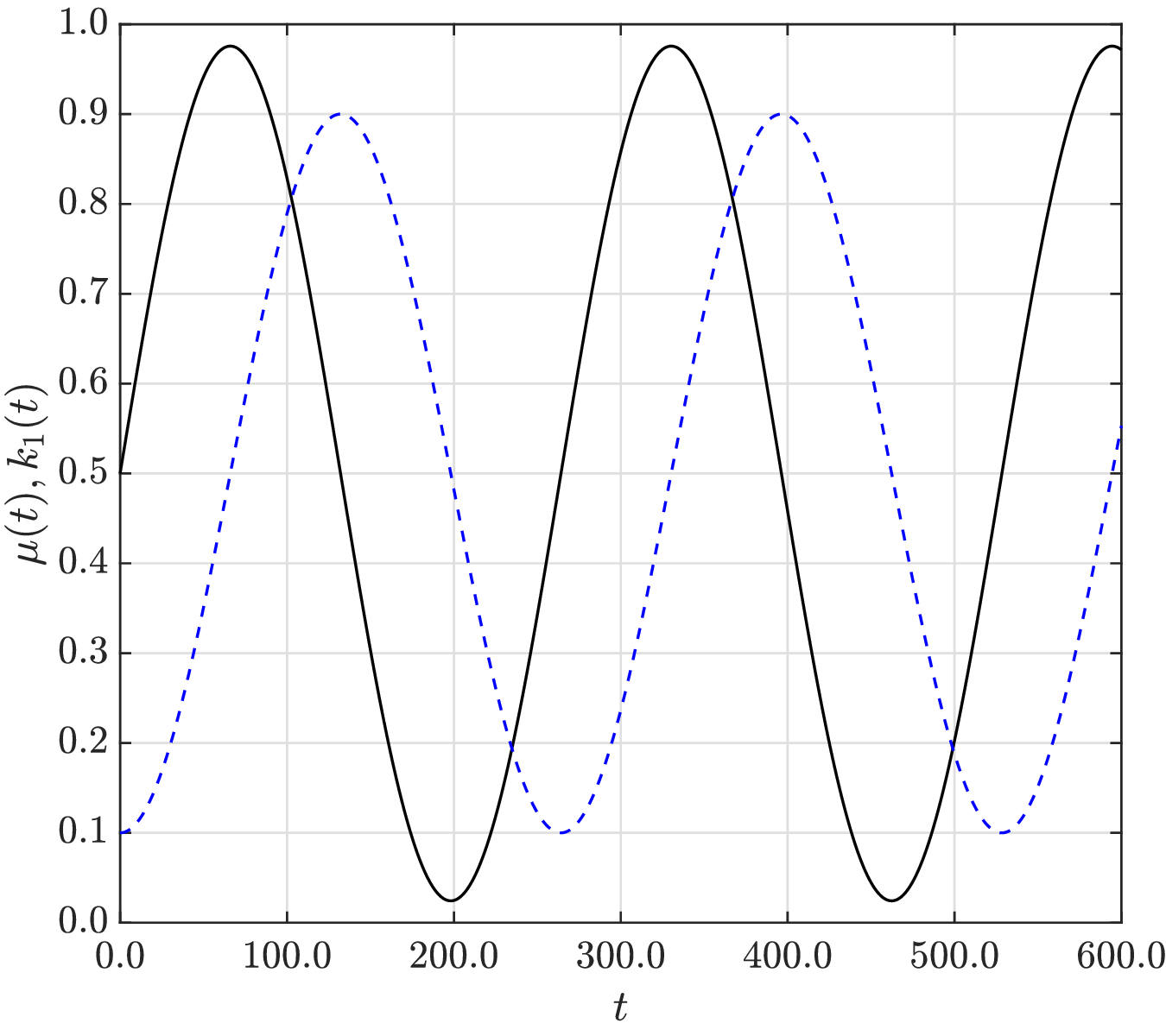}
        \end{subfigure}
        \caption{Stable breathing chimera state in a two-coevolutionary network. Local order parameters (left) and adaptive coupling strengths (right), respectively. First and second population synchronization level (orange, blue) and their filtered versions (red, green). Adaptive laws $\dot{\mu} = \varepsilon(-\sqrt{2}(k_{1} - 0.5))$ (black), $\dot{k}_{1} = \varepsilon(\mu - 0.5)$ (dashed blue), and initial conditions $(\rho_{1},\psi,\mu, k_{1}) = (0.9, 0.1, 0.5, 0.1)$. Parameters: $\Delta_{1}=0.5$, $\Delta_{2}=0.1$, $\omega_{1} = 101/20$, $\omega_{2} = \omega_{1}$, $k_{2}=1.0$, $\varepsilon=0.02$, \textcolor{black}{$\beta_{11}=\beta_{12}=\beta_{21}=\beta_{22}=0$}, and $N_{1,2}=500$. In this figure, even if  $k_1,\mu$ are both adaptive, thanks to the geometric understanding of the inter and intracoupling critical manifolds, we ensure that the corresponding critical manifold in this case is also normally hyperbolic. This geometric insight allows us to simply propose a periodic adaptation law that leads to this breathing chimera.}
        \label{Fig:BreathingTwoAdaptations}
    \end{figure*}

    \begin{figure*}[htbp]
        \begin{subfigure}[htbp]{0.49\textwidth}
            \includegraphics[width=1.0\linewidth]{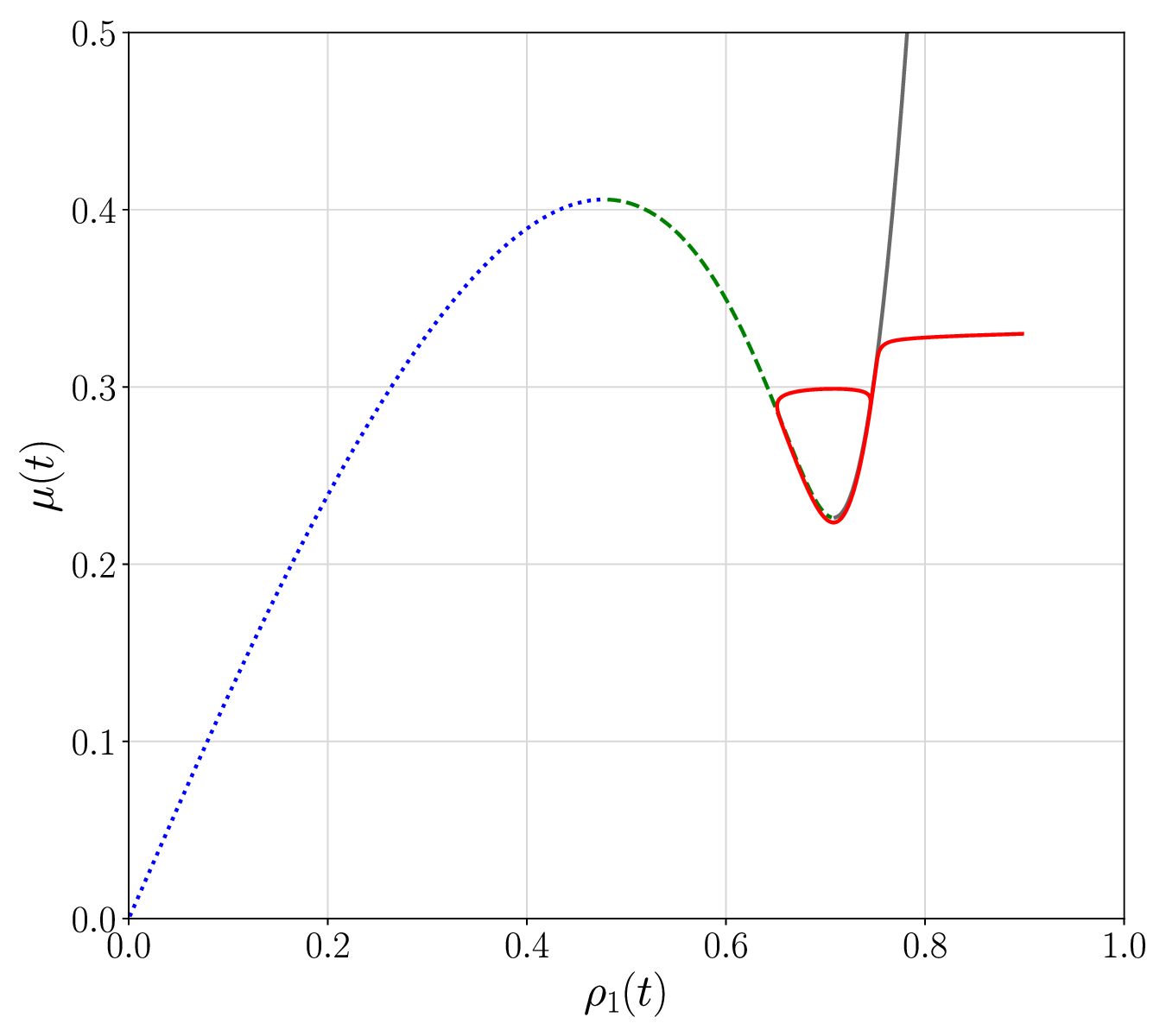}
        \end{subfigure}\hfill
        \begin{subfigure}[htbp]{0.5\textwidth}
            \includegraphics[width=1.0\linewidth]{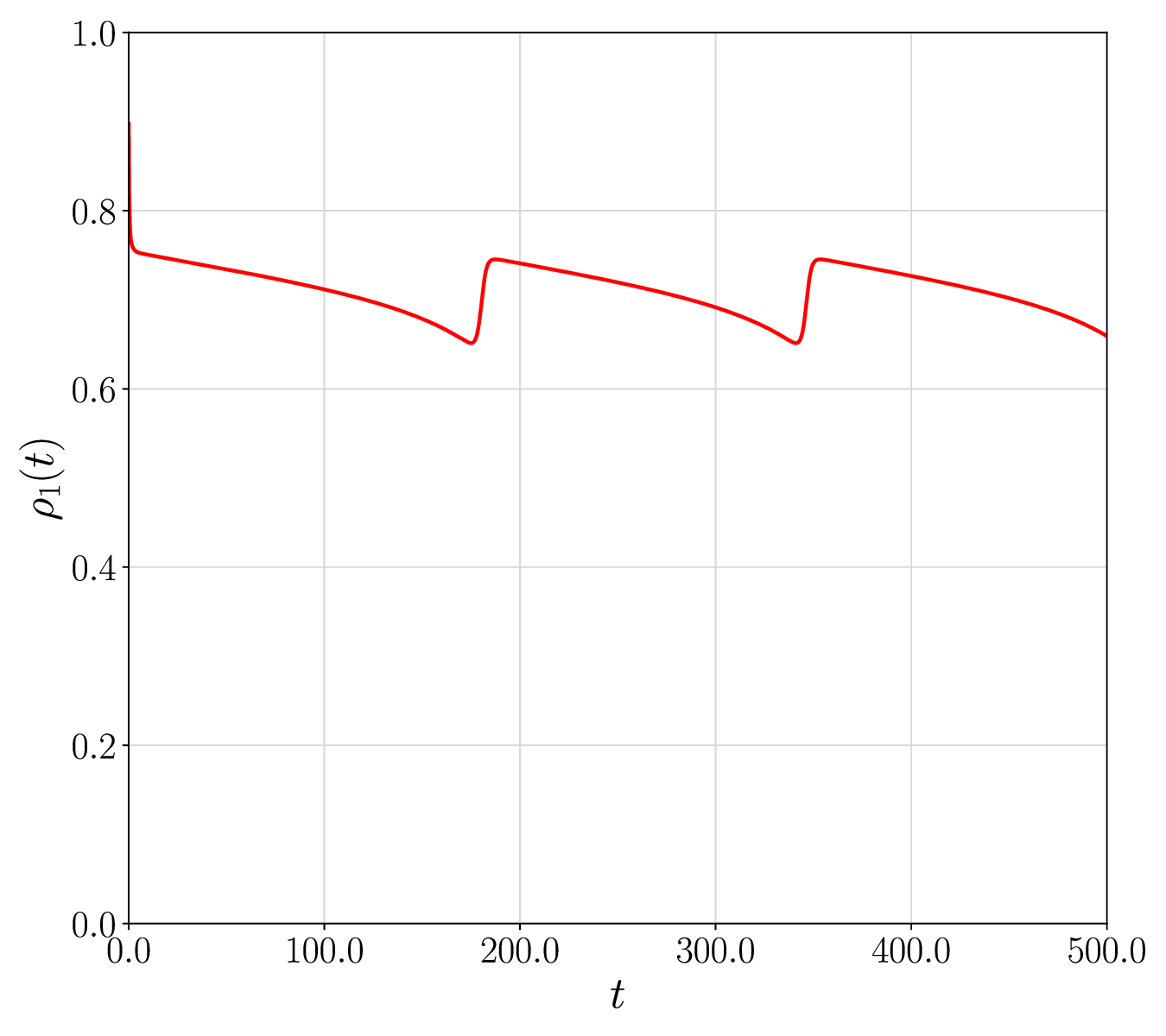}
        \end{subfigure}
        \caption{Canard cycle in a coevolutive intercoupling scenario. Projection to the $\left( \rho_{1},\mu \right)$ plane (left) and local order parameter $\rho_{1}(t)$ (right), respectively. Critical manifold regions: attracting (grey), saddle (dashed green), repelling (blue dotted), and first population synchronization level (red). Adaptive law $\dot{\mu} = \varepsilon_{\mu}(0.708065169053771149 - \rho_{1})$, and initial conditions $(\rho_{1},\psi,\mu) = (0.9, 0.1, 0.33)$. Parameters: $\Delta_{1}=0.5$, $\Delta_{2}=0.1$, $\omega_{1} = 0.0$, $\omega_{2} = 0.4$, $k_{1} = 2.0$, $k_{2}=3.0$, $\varepsilon_{\mu}=0.02$, and \textcolor{black}{$\beta_{11}=\beta_{12}=\beta_{21}=\beta_{22}=0$}.}
        \label{Fig:InterCanardHopf}
    \end{figure*}

    \begin{figure*}[htbp]
        \begin{subfigure}[htbp]{0.49\textwidth}
            \includegraphics[width=1.0\linewidth]{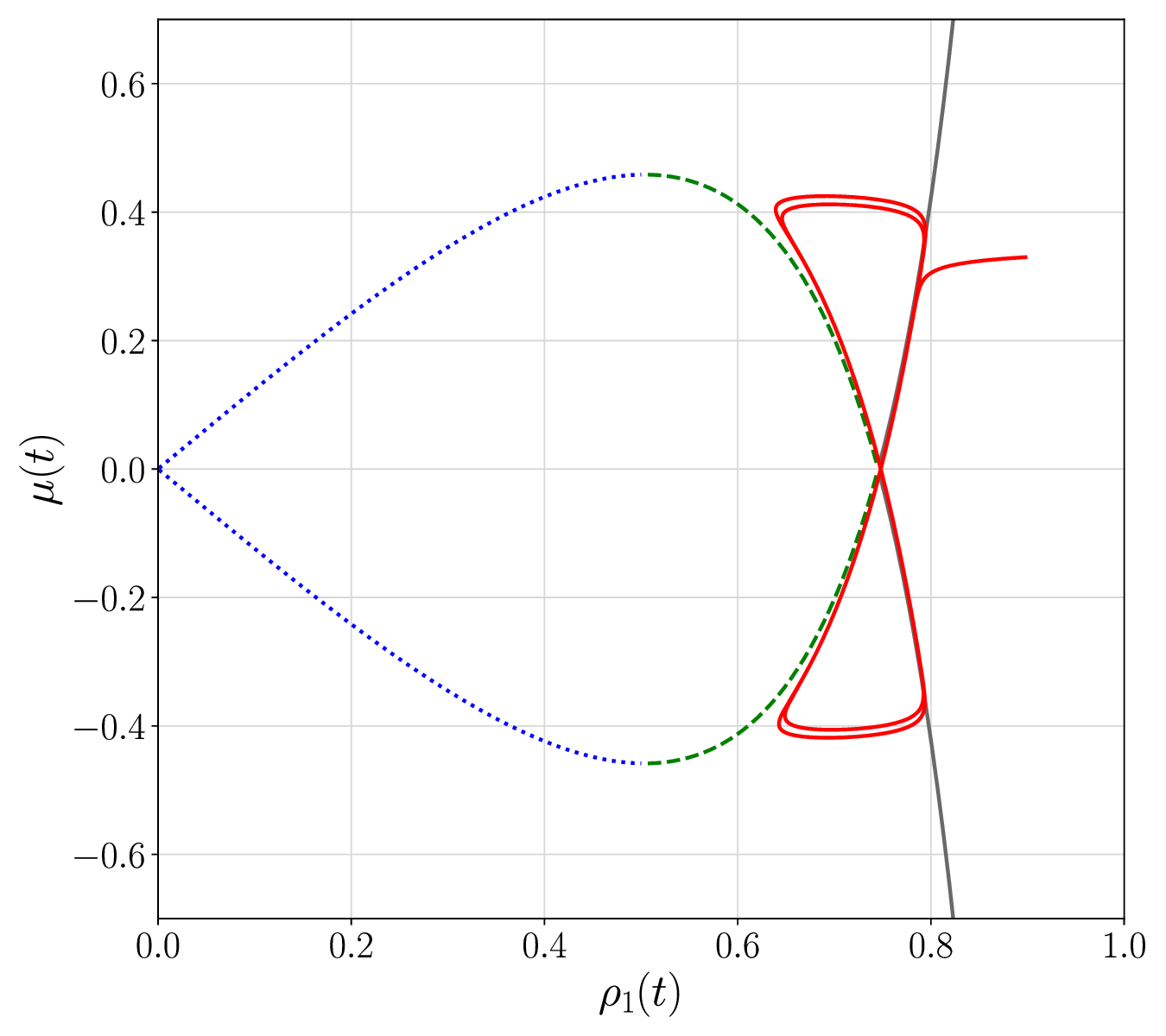}
        \end{subfigure}\hfill
        \begin{subfigure}[htbp]{0.5\textwidth}
            \includegraphics[width=1.0\linewidth]{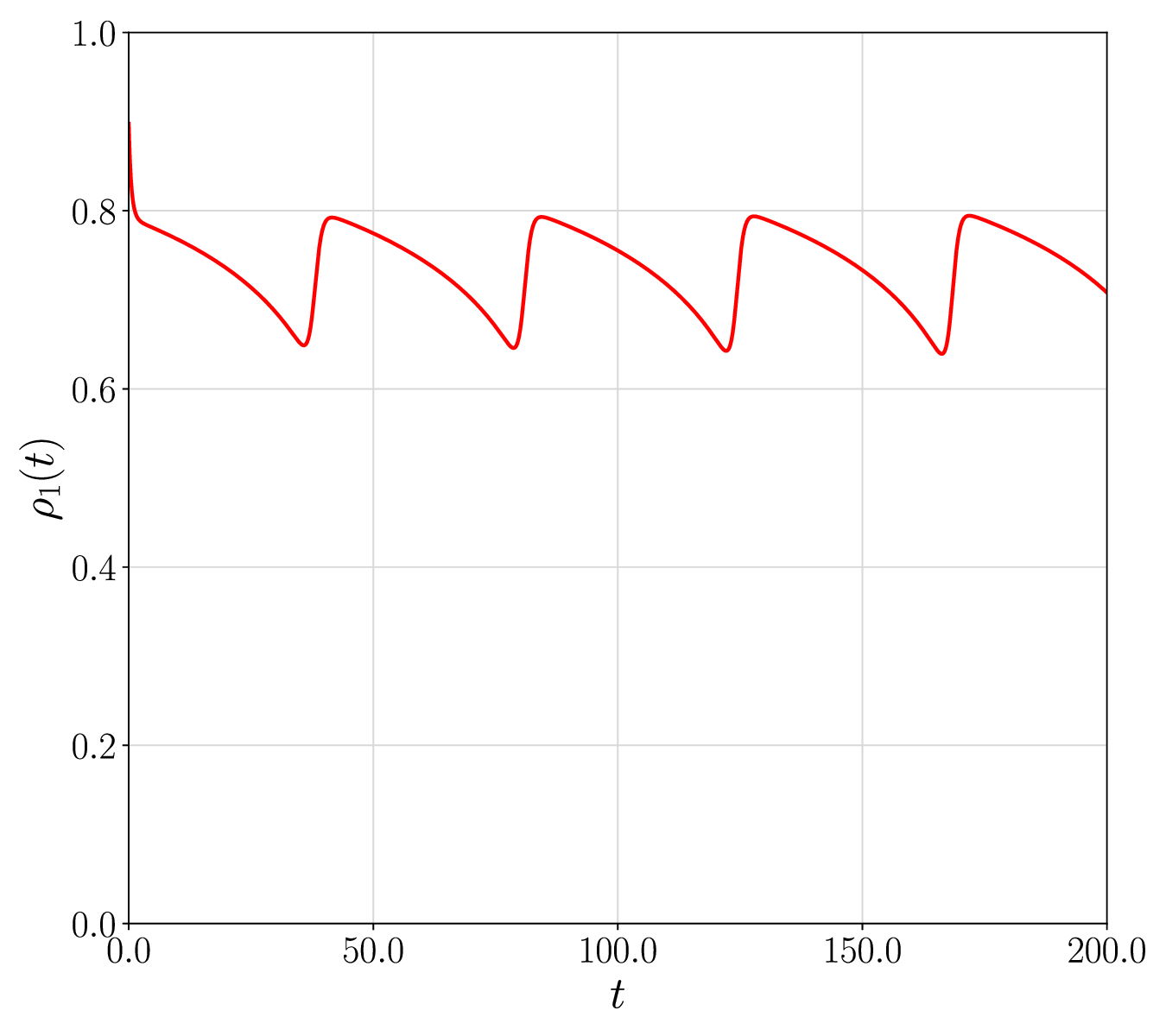}
        \end{subfigure}
        \caption{Transcritical canard cycle in the coevolutive intercoupling scenario. Projection to the $\left( \rho_{1},\mu \right)$ plane (left), and local order parameter $\rho_{1}(t)$ (right), respectively. Critical manifold regions: attracting (grey), saddle (dashed green), repelling (blue dotted), and first population synchronization level (red). Adaptive law $\dot{\mu} = -\varepsilon_{\mu}\cos{\psi}$, and initial conditions $(\rho_{1},\psi,\mu) = (0.9, 0.1, 0.33)$. Parameters: $\Delta_{1}=0.5$, $\Delta_{2}=0.1$, $\omega_{1} = 0.0$, $\omega_{2} = \omega_{1}$, $k_{1} = 2.25$, $k_{2}=3.0$, $\varepsilon_{\mu}=0.02$, and \textcolor{black}{$\beta_{11}=\beta_{12}=\beta_{21}=\beta_{22}=0$}.}
        \label{Fig:TranscriticalCanard}
    \end{figure*}

\appendix
\section{Preliminaries}
\label{Sec:Preliminaries}
In this section, we summarize the principal methods that are used through the development of our main results. First, we present the Ott-Antonsen ansatz which produces a mean-field equivalent system of the network under particular considerations. Additionally, we briefly introduce the fast-slow theory for singularly perturbed dynamics and the principal results here employed.

\subsection{Ott-Antonsen reduction}
In this work, we analyze a globally connected network formed by different types of heterogeneous Kuramoto oscillators with adaptive coupling strengths. We assume that the coevolutionary laws only depend on macroscopic quantities of the network, which allows us to treat the couplings as fixed parameters for the discussions in this section. Further details are given in section \ref{Sec:Model}. Therefore, we begin by studying an all-to-all network composed by different populations of heterogeneous Kuramoto phase oscillators, which can be expressed in the form
\textcolor{black}{
\begin{equation}
    \dot{\theta}_{i}^{\sigma} = \omega_{i}^{\sigma} + \sum_{\sigma' = 1}^{M} \frac{k_{\sigma \sigma'}}{N_{\sigma'}}\sum_{j=1}^{N_{\sigma'}} \sin{\left(\theta_{j}^{\sigma'}-\theta_{i}^{\sigma} - \beta_{\sigma \sigma'}\right)},
    \label{KuramotoMultilayer}
\end{equation}}
where $\theta_{i}^{\sigma}$ and $N_{\sigma}$ denote the state of the $i$-th oscillator and the number of elements in the $\sigma$-population, respectively. Additionally, $k_{\sigma \sigma'}$ and $\beta_{\sigma \sigma'}$ are the coupling strength and phase-lag between two layers, with $\sigma=1,2,\dots,M$. Moreover, $\omega_{i}^{\sigma}$ is the natural frequency of each oscillator in its respective community, selected from $M$ unimodal Cauchy-Lorentz distributions as
\begin{equation}
    g^{\sigma}(\omega_{\sigma}) = \frac{1}{\pi}\left[ \frac{\Delta_{\sigma}}{(\omega_{\sigma}-\hat{\omega}_{\sigma})^{2} + \Delta_{\sigma}^{2}} \right],
    \label{CauchyLorentzPDF}
\end{equation}
with $\hat{\omega}_{\sigma}$ and $\Delta_{\sigma}$ the central frequency and width parameter of each distribution. Notice that, the heterogeneity in the oscillators' natural frequency distributions itself can lead to the presence of chimera states. Therefore, in what follows we restrict our analysis to the case $\beta_{\sigma \sigma'}=0$. In fact, we will later see that, in our context, $\beta_{\sigma \sigma'}$ plays the role of a regular perturbation parameter for $\beta_{\sigma \sigma'}$ small. The case $\beta_{\sigma \sigma'}\approx\frac{\pi}{2}$ is relevant, but is left for future research because it leads to extra complications in the mean field analysis.

\paragraph*{}Considering the thermodynamic limit ($N_{\sigma}\rightarrow\infty$ \textcolor{black}{for each $\sigma$, but $\sigma$ finite}), the state of \eqref{KuramotoMultilayer} can be described by the continuous probability distribution functions $f^{\sigma}\left(\theta^{\sigma}, \omega_{\sigma}, t\right)$, which measure the fraction of oscillators with phases in the interval $[\theta^{\sigma}, \theta^{\sigma}+\textnormal{d}\theta^{\sigma}]$ and natural frequencies between $[\omega_{\sigma}, \omega_{\sigma}+\textnormal{d}\omega_{\sigma}]$ at time $t$. Moreover, since the number of oscillators is preserved for every population at all time, each density function satisfies the continuity equation
\begin{equation}
    \frac{\partial f^{\sigma}}{\partial t} + \frac{\partial}{\partial \theta^{\sigma}}\left( f^{\sigma}v^{\sigma} \right) = 0,
    \label{ContinuumEquationOttAntonsen}
\end{equation}
where $v^{\sigma}\left(\theta^{\sigma},\omega_{\sigma},  t\right)$ is the continuous version of the angular velocity of the oscillators in each population, expressed as

\begin{equation}
    v^{\sigma} = \omega_{\sigma} + \sum_{\sigma'=1}^{M} k_{\sigma \sigma'}\int_{0}^{2\pi}f\left( \theta^{\sigma'}, \omega_{\sigma'}, t \right)\sin{\left(\theta^{\sigma'} - \theta^{\sigma}\right)}\text{d}\theta^{\sigma'}.
    \label{Eq:AngularVelocity}
\end{equation}
Moreover, the well-known Kuramoto order parameter is defined as

\begin{equation}
    z_{\sigma}(t) = \rho_{\sigma}e^{\imath \phi_{\sigma}} = \int_{-\infty}^{\infty}\int_{0}^{2\pi}f\left( \theta^{\sigma}, \omega_{\sigma}, t \right)e^{\imath \theta^{\sigma}}\text{d}\theta^{\sigma}\text{d}\omega_{\sigma}.
    \label{Eq:KuramotoOrderParameter}
\end{equation}
In geometric terms, the complex order parameter describes the centroid of all the phasors $e^{\imath \theta^{\sigma}}$ and it grows larger as a higher synchronization level is reached in the network \cite{So2011}. Subsequently, it is possible to express the angular velocity \eqref{Eq:AngularVelocity} in terms of the order parameter \eqref{Eq:KuramotoOrderParameter}, resulting in
\textcolor{black}{
\begin{equation}
    v^{\sigma}\left( \theta^{\sigma}, \omega_{\sigma}, t \right) = \omega_{\sigma} + \frac{1}{2\imath}[ H_{\sigma \sigma'}(t)e^{-\imath \theta^{\sigma}} - \Bar{H}_{\sigma \sigma'}(t)e^{\imath \theta^{\sigma}} ],
    \label{Eq:AngularVelocityH}
\end{equation}}
where the over bar denotes the complex conjugate, and the functions $H_{\sigma \sigma'}(t)$ are any smooth, complex-valued $2\pi-$periodic functions of the phases $\theta_{i}^{\sigma}$, $i=1,\dots, N_{\sigma}$, \textcolor{black}{$\sigma = 1,2,\dots, M$} \cite{MarvelMirolloStrogatz2009}, and that in general may depend on the complex order parameter \eqref{Eq:KuramotoOrderParameter}. Moreover, as the exact time dependence of $H_{\sigma \sigma'}(t)$ does not affect the derivation of the Ott-Antonsen results, it suffices to consider $H_{\sigma \sigma'}(t)$ as some generic time dependent functions, regardless on how this dependence is determined \cite{OttAntonsen2009}. \textcolor{black}{In particular, for our problem consisting on a multilayer network of Kuramoto-type phase oscillators \eqref{GeneralModelLGVP} with coevolutive inter and intracouplings as \eqref{GenericAdaptiveLaw}, the functions $H_{\sigma \sigma'}(t)$ are defined as $H_{\sigma \sigma'}(t) = \sum_{\sigma'=1}^{M} k_{\sigma \sigma'} z_{\sigma'}(t)$, where the effective forcing depends on the population index $\sigma, \sigma' = 1,2$.} 
\paragraph*{}Following Ott and Antonsen \cite{OttAntonsen2008}, we focus our attention on a particular class of density functions $f^{\sigma}\left( \theta^{\sigma}, \omega_{\sigma}, t \right)$ in the form of Poisson kernels. In fact, Ott and Antonsen discovered that such kernels satisfy the governing equations if an associated low-dimensional system of ordinary differential equations is satisfied \cite{Abrams2008}. Thus, by considering Fourier series of the form
\begin{eqnarray}
    f^{\sigma}\left( \theta^{\sigma}, \omega_{\sigma}, t \right) = \frac{g^{\sigma}(\omega)}{2\pi}\Big[ 1 + \sum_{n=1}^{\infty}\alpha_{\sigma}^{n}(\omega_{\sigma}, t)e^{\imath n\theta^{\sigma}}
    + \sum_{n=1}^{\infty}\bar{\alpha}^{n}_{\sigma}(\omega_{\sigma}, t)e^{-\imath n\theta^{\sigma}}\Big],
    \label{OttAntonsenAnsatz}
\end{eqnarray}
for the density function of each population, and by imposing the ansatz \eqref{OttAntonsenAnsatz} to the continuum equation \eqref{ContinuumEquationOttAntonsen} we obtain the following set of differential complex equations
\textcolor{black}{
\begin{eqnarray}
    0=\sum_{n=1}^{\infty}n\alpha_{\sigma}^{n-1}e^{\imath n\theta^{\sigma}}\Big(\dot{\alpha}_{\sigma}+ \imath\omega_{\sigma}\alpha_{\sigma} + \frac{1}{2}\big[ \alpha_{\sigma}^{2}H_{\sigma \sigma'}(t) - \bar{H}_{\sigma \sigma'}(t)\big]\Big) + \textnormal{c.c.},
    \label{OttAntonsenAnsatzReduced}
\end{eqnarray}}\par
\noindent where c.c. stands for the complex conjugate of the expression in the right-hand side of \eqref{OttAntonsenAnsatzReduced}. Moreover, since the sum component in \eqref{OttAntonsenAnsatzReduced} is not zero, it means that the term within the parentheses vanish identically if the ansatz \eqref{OttAntonsenAnsatz} is a valid solution to the continuity equation \eqref{ContinuumEquationOttAntonsen}. Hence,
\textcolor{black}{
\begin{equation}
    \dot{\alpha}_{\sigma} + \imath\omega_{\sigma}\alpha_{\sigma} + \frac{1}{2}\left[ \alpha_{\sigma}^{2}H_{\sigma \sigma'}(t) - \Bar{H}_{\sigma \sigma'}(t) \right] = 0,
    \label{OttAntonsenMeanField}
\end{equation}}\par
\noindent with $z_{\sigma}(t) = \int_{-\infty}^{\infty}\alpha_{\sigma}(\omega_{\sigma},t)g^{\sigma}(\omega_{\sigma})\text{d}\omega_{\sigma}$. Furthermore, consider solutions of \eqref{OttAntonsenMeanField} with the initial conditions $\alpha_{\sigma}(\omega_{\sigma},0)$ that satisfy: (i) $|\alpha_{\sigma}(\omega_{\sigma},t)|\leq1$; (ii) $\alpha_{\sigma}(\omega_{\sigma},0)$ is analytically continuable into the lower half plane $\text{Im}(\omega_{\sigma})<0$; and (iii) $|\alpha_{\sigma}(\omega_{\sigma},t)|\rightarrow 0$ as $\text{Im}(\omega_{\sigma})\rightarrow-\infty$. If such conditions are satisfied for $\alpha_{\sigma}(\omega_{\sigma}, 0)$, then they will continue to be satisfied for $\alpha_{\sigma}(\omega_{\sigma}, t)$ \cite{OttAntonsen2008, Martens2009}. It is important to mention that the Ott-Antonsen ansatz corresponds to a particular case of the technique developed by Watanabe-Strogatz \cite{Watanabe1993, Watanabe1994} which considers all the harmonics of the frequency distribution function \eqref{OttAntonsenAnsatz}. Thus, although the resulting response in the Ott-Antonsen manifold does not entirely capture the transient behavior of the synchronization level in the network, the stationary states of both systems are effectively the same \cite{Pikovsky2011}. Furthermore, the Ott-Antonsen manifold is known to be globally attracting when the distribution of natural frequencies is non-homogeneous, i.e., $\Delta_{\sigma}>0$ \cite{OttAntonsen2009, OttAntonsen2011}. Therefore, every long time behavior of the order parameter can be obtained by analyzing the corresponding mean-field even when weak heterogeneities are considered \cite{Lee2021}.
\paragraph*{}By selecting the probability distribution functions as \eqref{CauchyLorentzPDF}, it is possible to calculate $\bar{z}_{\sigma}$ by the contour integration in the negative half of the complex plane yielding $\bar{z}_{\sigma} = \alpha_{\sigma}\left(\omega_{\sigma} - \imath\Delta_{\sigma},t \right)$ \cite{Jalan2022}. Substituting this expression in \eqref{ContinuumEquationOttAntonsen} produces $M$-coupled complex ODEs describing the evolution of the complex order parameter as
\begin{equation*}
    \dot{z}_{\sigma} + \left( \Delta_{\sigma} - \imath\omega_{\sigma} \right)z_{\sigma} + \frac{1}{2}\sum_{\sigma' = 1}^{M} k_{\sigma \sigma'}\left( \bar{z}_{\sigma'}z_{\sigma}^{2} - z_{\sigma'} \right) = 0.
    \label{ComplexMeanField}
\end{equation*}
Finally, with $z_{\sigma} = \rho_{\sigma}e^{-\imath\phi_{\sigma}}$, for which the negative sign is included in the definition of $\phi_{\sigma}$ in order for the Poisson kernel to converge to $\delta(\theta-\phi_{\sigma})$ and not to $\delta(\theta+\phi_{\sigma})$ as $\rho_{\sigma}\rightarrow1$ \cite{Abrams2008}, produces by orthogonality $2M$ coupled real ODEs in the form
\begin{align}
    \begin{split}
        \dot{\rho}_{\sigma} &= -\Delta_{\sigma}\rho_{\sigma} + \frac{1}{2}\left( 1-\rho_{\sigma}^{2} \right)\sum_{\sigma'=1}^{M}k_{\sigma \sigma'}\rho_{\sigma'}\cos{\left( \phi_{\sigma'} - \phi_{\sigma} \right)},\\
        \dot{\phi}_{\sigma} &= -\omega_{\sigma} + \frac{1}{2}\frac{\rho_{\sigma}^{2}+1}{\rho_{\sigma}}\sum_{\sigma'=1}^{M}k_{\sigma \sigma'}\rho_{\sigma'}\sin{\left( \phi_{\sigma'}-\phi_{\sigma} \right)}.
    \end{split}
    \label{OttAntonsenRealMeanField}
\end{align}
In particular, our mechanism for the generation of stable chimera state relies on geometric properties of \eqref{OttAntonsenRealMeanField} with adaptive coupling strengths under a macroscopic slow adaptation, which induces a time-scale separation between the dynamics on and off the network. Thus, the next subsection briefly summarizes the principal concepts of fast-slow systems employed later on.

\subsection{Fast-slow dynamics}
    A fast-slow system is a singularly perturbed ordinary differential equation in the form
    \begin{align}
    \begin{split}
        \varepsilon\dot{x} &= f(x,y,\varepsilon),\\
        \enspace \dot{y} &= g(x,y,\varepsilon),
        \label{SlowTimeForm}
    \end{split}
    \end{align}
    where the over-dot represents the derivative with respect to the slow-time $\tau$, $x\in\mathbb{R}^{m}$ and $y\in\mathbb{R}^{n}$ denote the fast and slow variables respectively, and $0<\varepsilon\ll1$ describes the time-scale separation. Moreover, we consider $f:\mathbb{R}^{m}\times\mathbb{R}^{n}\times\mathbb{R}\rightarrow \mathbb{R}^{m}$ and $g:\mathbb{R}^{m}\times\mathbb{R}^{n}\times\mathbb{R}\rightarrow \mathbb{R}^{n}$ of class $C^{k}$, for $k$ sufficiently large. Equivalently, by defining the fast time $t\coloneqq\tau/\varepsilon$, we can rewrite \eqref{SlowTimeForm} as
    \begin{align}
    \begin{split}
        x' &= f(x,y,\varepsilon),\\
        y' &= \varepsilon g(x,y,\varepsilon).
        \label{FastTimeForm}
    \end{split}
    \end{align}
    where the prime denotes the derivative with respect to the fast time $t$. Hence, \eqref{SlowTimeForm} and \eqref{FastTimeForm} are known as the slow and fast formulations, and in the singular limit ($\varepsilon=0$) two different problems arise, namely the reduced and the layer problem, respectively from \eqref{SlowTimeForm} and \eqref{FastTimeForm}. Despite of being non-equivalent both problems share a close connection as expressed in the following definition.
\begin{definition}[Critical manifold]
\label{CriticalManifoldDefinition}
    The critical set is defined as
    \begin{equation}
    \label{CriticalManifoldDef}
         \mathcal{C}_{0} = \left\{(x,y) \in \mathbb{R}^{m}\times \mathbb{R}^{n}: f(x,y,0)=0 \right\}.
    \end{equation}
\end{definition}
Moreover, if $\mathcal{C}_{0}$ is a submanifold of $\mathbb{R}^{m}\times\mathbb{R}^{n}$, it is referred to as the \emph{critical manifold}. Notice that points of \eqref{CriticalManifoldDef} are in a direct correspondence to the equilibria set of the fast flow, generated by \eqref{FastTimeForm}. Additionally, a relevant property that $\mathcal{C}_{0}$ can have is normal hyperbolicity, given as follows.
\begin{definition}[Normal hyperbolicity]
    \label{DefinitionNormalHyperbolicity}
    A subset $\mathcal{S}\subset\mathcal{C}_{0}$ is normally hyperbolic if the matrix $\left( \textnormal{D}_{x}f \right)(p,0)$ has no eigenvalues with zero real part for all $p\in\mathcal{S}$. Moreover, a normally hyperbolic subset is  attracting (repelling) if all the eigenvalues of $\left( \textnormal{D}_{x}f \right)(p,0)$ have negative (positive) real part for every $p\in\mathcal{S}$. Finally, if $\mathcal{S}$ is normally hyperbolic but neither attracting nor repelling it is of the saddle type.
\end{definition}

In contrast, $\mathcal{S}$ is non-hyperbolic if the corresponding matrix $\left( \textnormal{D}_{x}f \right)(p,0)$ has at least one eigenvalue with zero real part for any $p\in\mathcal{S}$. Depending on the normal hyperbolicity condition of \eqref{CriticalManifoldDef} different analysis techniques can be employed. For instance, non-hyperbolic points, related to dynamic features such as relaxation oscillations and canards \cite{Desroches2016}, can be studied through the blow-up method \cite{Jardon2021}. On the other hand, for normally hyperbolic cases we have Fenichel's Theorem \cite{Fenichel1979, Jones1995}, stated as follows. 
    \begin{theorem}[Fenichel's Theorem]
        \label{FenichelTheorem}
        Given a compact normally hyperbolic submanifold (possibly with boundary) $\mathcal{S}=\mathcal{S}_{0}$ of the critical manifold $\mathcal{C}_{0}$ of \eqref{SlowTimeForm}, and $f,g \in C^{k<\infty}$, for $0<\varepsilon\ll1$, then:
        \begin{enumerate}
            \item A locally invariant manifold $\mathcal{S}_{\varepsilon}$ diffeomorphic to $\mathcal{S}_{0}$ exists. Local invariance means that trajectories can enter or escape $\mathcal{S}_{\varepsilon}$ only through its boundaries.
            \item $\mathcal{S}_{\varepsilon}$ has Hausdorff distance $\mathrm{d_{H}}\left( V,W \right) = O\left(\varepsilon\right)$ from $\mathcal{S}_{0}$, as $\varepsilon \rightarrow 0$.
            \item The flow on $\mathcal{S}_{\varepsilon}$ converges to the flow generated by the reduced problem, known as the slow flow, as $\varepsilon\rightarrow0$.
            \item $\mathcal{S}_{\varepsilon}$ is $C^{k}$-smooth.
            \item $\mathcal{S}_{\varepsilon}$ is normally hyperbolic and has the same stability properties with respect to the fast variables as $\mathcal{S}_{0}$.
            \item $\mathcal{S}_{\varepsilon}$ is usually not unique. In regions with a fixed distance from $\partial \mathcal{S}_{\varepsilon}$, all manifolds which satisfies conditions 1-5 above are exponentially close to each other with a Hausdorff distance of order $O\left(\textnormal{exp}(-C/\varepsilon)\right)$ for some $C>0$, $C\in O(1)$, as $\varepsilon\rightarrow0$. 
        \end{enumerate}
        Thus, any manifold $\mathcal{S}_{\varepsilon}$ satisfying  conditions 1-5 is called a slow manifold.
    \end{theorem}
    \begin{remark}
        Regularly, a particular manifold which satisfies conditions $1-5$ of theorem \ref{FenichelTheorem} is referred as \emph{the} slow manifold. Rigorously, this is not correct since $\mathcal{S}_{\varepsilon}$ is not unique, but since all possible choices are exponentially close, it is arbitrary to select one of these manifolds for most analytical and numerical considerations. Moreover, normally hyperbolic submanifolds of the slow problem, such as equilibria and periodic orbits, persist in every slow manifold  \cite{Kuehn2015}.
    \end{remark}
    In the past sometimes Fenichel's Theorem was referred to as \textit{Geometric Singular Perturbation Theory} (GSPT). Nevertheless, nowadays GSPT includes several geometric techniques such as fast-slow normal form theory and the blow-up method, to name a few.

\section*{References}

\printbibliography[heading=none]

\end{document}